\documentclass[12pt,a4paper,oneside,reqno]{article}

\usepackage[utf8]{inputenc}
\usepackage[english]{babel}
\usepackage[left=1in,right=1in,top=1in,bottom=1in]{geometry}
\usepackage{amsfonts,amssymb,amsmath,amsthm}
\usepackage[symbol]{footmisc}
\usepackage{longtable,array}
\newcolumntype{L}[1]{>{\raggedright\arraybackslash$}p{#1}<{$}}
\usepackage{enumerate}
\usepackage{cite}
\usepackage{xcolor}

\usepackage{hyperref}
\hypersetup{hidelinks}
\usepackage{fancyhdr}
\theoremstyle{plain}
\newtheorem{theorem}{Theorem}
\newtheorem{lemma}[theorem]{Lemma}
\newtheorem{proposition}[theorem]{Proposition}
\newtheorem{corollary}[theorem]{Corollary}
\theoremstyle{definition}
\newtheorem{remark}[theorem]{Remark}
\newtheorem{definition}[theorem]{Definition}

\makeatletter
\@removefromreset{theorem}{section}\@removefromreset{equation}{section}

\makeatother
\usepackage{adjustbox}
\usepackage{tikz}\usetikzlibrary{arrows}

\IfFileExists{fouriernc.sty}{\usepackage{fouriernc}}{}

\newcommand{\doi}[1]{\href{https://doi.org/#1}{doi:\detokenize{#1}}}
\newcommand{\arxiv}[1]{\href{https://arxiv.org/abs/#1}{arXiv:#1}}
\newcommand{\arxivsearch}[1]{\href{https://arxiv.org/search/?query=#1}{arXiv}}

\def\SAs{\mathrm{SAs}}
\def\CAs{\mathrm{CAs}}

\begin{document}

\noindent\fbox{\parbox{\dimexpr\textwidth-2\fboxsep-2\fboxrule\relax}{\small
\textbf{Corrected version (verified by Claude).} This is a corrected version of the paper H. Abdelwahab, I. Kaygorodov, B. Sartayev,
\emph{Shift associative algebras},
Rendiconti del Circolo Matematico di Palermo. Second Series, 74 (2025), no.~5, 145,
\doi{10.1007/s12215-025-01261-1}.}}

\bigskip

\begin{center}
{\Large Shift associative algebras}

\bigskip

{\bf Hani Abdelwahab, Ivan Kaygorodov, and Bauyrzhan Sartayev}
\end{center}

\bigskip

\noindent{\small\textbf{Abstract.} We present a comprehensive study of algebras satisfying the identity $(xy)z=y(zx)$, which we call shift associative algebras.
We show that these algebras are related to many interesting identities; in particular, they are related to anti-Poisson--Jordan algebras
and to algebras of associative type $\sigma$.
We determine which operads of algebras of associative type $\sigma$ are Koszul and which are self-dual.
A basis of the free shift associative algebra generated by a countable set is constructed.
An analog of the Wedderburn--Malcev theorem is established.
The algebraic and geometric classifications of complex shift associative algebras of dimension at most $4$ are given.
In particular, we prove that the minimal dimension of a non-associative shift associative algebra is $5$.}

\bigskip
\bigskip


\bigskip


\bigskip

\bigskip

\tableofcontents

\bigskip

\newpage
\section*{Introduction}\addcontentsline{toc}{section}{Introduction}

We say that an algebra is of associative type $\sigma \in \mathbb{S}_3$ if it (or its opposite algebra) satisfies an
identity of one of the following types
\begin{center}
    $(x_{1}x_2)x_3=x_{\sigma(1)}(x_{\sigma(2)}x_{\sigma(3)})$ \ 
    or 
    \  $(x_{1}x_2)x_3=(x_{\sigma(1)}x_{\sigma(2)})x_{\sigma(3)}.$ \ 
\end{center}
We say that an algebra has the first (resp., second) type if it satisfies the first (resp., second) above-mentioned identity.

This class includes many important varieties of algebras, which are actively studied
(see, for example, recent papers on Novikov algebras by
Dotsenko, Ismailov, Kolesnikov, Sartayev, Umirbaev, and others
\cite{diu,B,SK23};
on bicommutative algebras by
Drensky, Ismailov, Sartayev, Zhakhayev, and others \cite{dimz,IMS24};
on perm algebras by Kaygorodov, Mashurov, and Sartayev \cite{MS24,KM24}).
On the other hand,
the study of algebras of associative type has a long history.
Widiger considered algebras of the first associative type with $\sigma=(12)$ in \cite{W02};
Felipe considered algebras of the second associative type with $\sigma=(12)$ in \cite{F11, FR11};
Kleinfeld considered algebras of the second associative type with $\sigma=(13)$ in \cite{K78}. 
There are some works about groupoids with identities of associative type \cite{JK,PB05}.
The most popular case of algebras of associative type is $\sigma=(123).$ Algebras of the second associative type
with $\sigma=(123)$ were considered in 
papers by Kleinfeld \cite{K95};
by Behn, Correa, and Hentzel in  \cite{B08,Behn}; 
Dhabalendu  and   Hentzel in \cite{DH};  
and by  Dhabalendu in  \cite{D23}.
Finally, various authors considered algebras of the first associative type
with $\sigma=(123)$

(Chen in \cite{CJJ};
Hentzel, Jacobs, and Muddan in \cite{Hentzel};
Widiger  in \cite{W02};
Dassoundo and Silvestrov in \cite{Sergei};
Barreiro, Benayadi, and Rizzo in \cite{BBR}).
In the present paper, we continue the study of algebras of the first associative type
with $\sigma=(123),$ which we call {shift associative algebras}, and of their associative subvariety, whose members we call
cyclic associative algebras.

\begin{definition}
An algebra is called a shift associative algebra\footnote{This type of algebras was introduced under the name of {\it nearly associative} algebras in \cite{BBR,BBR2,Sergei}.
Since the term ``nearly associative'' is also used for various other classes of algebras \cite{BG96,K86,sh58}, which are not related to the
algebras introduced in \cite{BBR,BBR2,Sergei}, and especially in the famous book \cite{sh82}, we have decided to use another name.} if it satisfies the
identity
\begin{equation}
(x y)  z \ =\ y ( zx).  \label{nearly_eq}
\end{equation}
\end{definition}

\begin{definition}
An algebra is called a {cyclic associative algebra} if it satisfies the identities
\begin{equation*}
(x y)  z \ =\  x(yz)\ = \ y(zx).
\end{equation*}
\end{definition}

It is clear that 
any commutative associative algebra is a shift associative algebra; 
any commutative shift associative algebra is associative;
and any cyclic associative algebra is shift associative.

\ 

\noindent
{\bf Notations.}
We omit some well-known definitions
(of Lie algebras, Jordan algebras, nilpotent algebras, solvable algebras, idempotents, etc.) and refer the reader to previously published papers (for example, \cite{BBR,BBR2,Sergei,Albert,sh82}).
The ground field has characteristic
different from two. For the commutator, the anticommutator, and the associator,
we use the standard notation

\begin{center}
    $[x,y] : = \frac  12(xy-yx),$   \  
    $x \circ y : = \frac  12 (xy+yx)$ \ and \ 
    $(x,y,z):=(xy)z-x(yz).$
\end{center}



\section{Identities related to shift associative  algebras}

Throughout this section, the ground field has characteristic different from $2$.
The next proposition gives a generalization of 
\cite[Proposition 3.7 and Corollary 3.11]{BBR}.

\begin{proposition}
    Let $\mathcal{A}$ be an associative algebra which is of associative type
    $\sigma_1=(1i)$ and of associative type $\sigma_2=(j3)$, where
    $i\neq1$, $j\neq 3$, and
    $\sigma_1 \neq \sigma_2$,
    or an associative algebra of associative type $\sigma=(132).$
    Then $\mathcal{A}$ is a cyclic associative algebra.  

\end{proposition}

\begin{proof}
    The statement follows from the relations on $\mathbb S_3$ given below:
    \begin{center}
    $(132)(132)\ =\ (12)(13)\ =\ (23)(12)\ =\ (13)(23) \ =\ (123).$ 
    \end{center}
\end{proof}

\begin{proposition}[see, \cite{BBR}]
    Let $\mathcal{A}$ be a shift associative algebra. Then $\mathcal{A}$ satisfies the following identities

\begin{equation}
(x  y)  \left( z  t\right) \ =\ \left( y  x\right)   \left(
t  z\right);  \label{swap}\\
\end{equation}
\begin{equation}
z \circ [x, y]  + y \circ  [x, z] \  =\ [z,x \circ y]+ [y,x \circ z].\label{blabla}\\
\end{equation}

\end{proposition}

\begin{definition}
An algebra $({\rm P}, \cdot, \{\cdot,\cdot\})$ is called an {anti-Poisson algebra}\footnote{
A similar notion of anti-Poisson algebras can be found in \cite{R22},
where an algebra $({\rm P}, \cdot, \{\cdot,\cdot\})$
is called an anti-Poisson algebra if  
$({\rm P}, \cdot)$ is mock-Lie, 
$({\rm P}, \{\cdot,\cdot\})$ is anticommutative, and (\ref{antipois}) holds.
} (resp., an anti-Poisson--Jordan algebra)
if $({\rm P}, \cdot)$ is a  commutative associative  (resp., Jordan) algebra,
$({\rm P},  \{\cdot,\cdot\})$ is a Lie algebra, and the following
identity holds:
\begin{equation}\label{antipois}
\{x  \cdot y, z \} = -\big(x   \cdot  \{ y,z \} +\{x,z\}  \cdot  y\big).
\end{equation}

\end{definition}

\begin{theorem}
\label{Jor-admiss}
Let $\mathcal{A}$ be a shift associative algebra. Then $(\mathcal{A}, \circ, [\cdot,\cdot])$ is an  anti-Poisson-Jordan algebra satisfying 
\begin{equation}
[x, y \circ z] + [z,x \circ y]+ [y,x \circ z]=0. \label{tipapois}
\end{equation}
$(\mathcal{A}, \circ, [\cdot,\cdot])$ is an
anti-Poisson algebra if and only if $\mathcal{A}$ satisfies the identity $[[x,y],z]=0.$
\end{theorem}

\begin{proof}
By \cite[Proposition 2.5]{Sergei},
$(\mathcal A, [\cdot, \cdot])$ is a Lie algebra for each shift associative algebra $\mathcal A$.
By \cite[Proposition 4.2]{BBR},
$(\mathcal A, \circ)$ is a Jordan algebra for each shift associative algebra $\mathcal A$.
 
Let us now check the relation (\ref{antipois}):

{\setlength{\tabcolsep}{2pt}\begin{longtable}{rclcl}
$x \circ [y,z]$ &$=$&$\frac  14 \big(x(yz)-x(zy)+(yz)x-(zy)x \big)$&$=$&$\frac  14 \big(x(yz)- x(zy)+z(xy)-y(xz) \big)$,\\
$ y\circ [x,z]$ &$=$&$\frac  14 \big(y(xz)-y(zx)+(xz)y-(zx)y \big)$&$=$&$\frac  14 \big(y(xz)- y(zx)+z(yx)-x(yz) \big)$,\\
$  [x \circ y,z]$ &$=$&$ \frac  14\big((xy)z-z(xy)+(yx)z-z(yx) \big)$&$=$&$ \frac  14\big(y(zx)-z(xy)+x(zy)-z(yx) \big)$.\\
\end{longtable}}
Summing up these three relations, we obtain that $(\mathcal{A}, \circ, [\cdot,\cdot])$ satisfies (\ref{antipois}); hence, it is an anti-Poisson--Jordan algebra.
By (\ref{blabla}) and (\ref{antipois}), we obtain (\ref{tipapois}).

If $(\mathcal{A}, \circ, [\cdot,\cdot])$ is an anti-Poisson algebra, 
then

\begin{longtable}{rclclc}
$0$&$=$&$4 \big((x \circ y) \circ z- x \circ (y \circ z)\big) \ =$\\ 
&$=$&$ (xy)z+z(xy)+(yx)z+z(yx)-x(yz)-x(zy)-(yz)x-(zy )x$ &$=$ \\
&$=$&$ y(zx)+z(xy)+x(zy)+z(yx)-x(yz)-x(zy)-z(xy)-y(xz)$ &$=$ \\
&$=$&$  (xz)y-y(xz) +y(zx)-(zx)y \ = \ 4\ [[x,z],y].$  \\

\end{longtable}
\end{proof}

\begin{corollary}
        Let $\mathcal A$ be a cyclic associative algebra.
    Then $(\mathcal A, \circ, [\cdot, \cdot])$ is both an anti-Poisson algebra and a Poisson algebra.
\end{corollary}

\begin{corollary}
\label{idem}
        Any finite-dimensional shift associative algebra which is not a nilalgebra contains an idempotent $e\neq 0$.
\end{corollary}

\begin{proof}
By Lemma \ref{pa} below, every shift associative algebra is power-associative, and each finite-dimensional power-associative algebra which is not a nilalgebra contains a nonzero idempotent \cite{Albert}.
\end{proof}

\begin{lemma}\label{pa}
Let $\mathcal{A}$ be a shift associative algebra. Then $\mathcal{A}$ is
power-associative. Hence, it is a mono-associative algebra.
\end{lemma}

\begin{proof}
By (\ref{nearly_eq}), we have $\left( x,x,x\right) =0.$
Since $(\mathcal{A}, \circ)$ is a Jordan algebra (Theorem \ref{Jor-admiss}), we have $\left( x^{2},x,x\right) =0$. Hence, by \cite[Lemma 3]{Albert}, $\mathcal{A}$ is power-associative.
\end{proof}
\begin{lemma}
\label{unital}
Each unital shift associative algebra  $\mathcal{A}$  is
commutative associative.
\end{lemma}
\begin{proof}
 Putting $z=1$ in $\left( \ref{nearly_eq}\right),$ we obtain $xy=  yx$ for all $x,y \in \mathcal{A}$, i.e., $\mathcal{A}$ is commutative. Since every commutative shift associative algebra is associative, $\mathcal{A}$ is associative.
\end{proof}

\begin{corollary}
Let $\mathcal{A}$ be a shift associative algebra. If its {unital hull} $\widehat{\mathcal{A}} := \langle 1\rangle \oplus \mathcal{A}$ with product
$$(\alpha 1+x)   (\beta 1+y) := \alpha \beta 1 + (\alpha y + \beta x+x   y),$$ is also  a shift associative  algebra, then 
$\mathcal{A}$  is commutative associative.
\end{corollary}

\begin{definition}\label{knice}
For a positive integer $k$, an algebra $\mathcal{A}$ is called {$k$-nice} if
the product of any $k$ elements of $\mathcal{A}$ does not depend on their association and order.
\end{definition}
Following the ideas of \cite[Theorem 2.1]{Hentzel},
we obtain the following result.

\begin{theorem}\label{k-nice}
Every shift associative algebra is $k$-nice for each $k\geq5$,
every cyclic associative algebra is $k$-nice for each $k\geq4$,
and every commutative associative algebra is $k$-nice for each $k\geq2$.
Moreover, the free shift associative algebra is not $4$-nice, and the free cyclic associative algebra is not $3$-nice.
\end{theorem}

\begin{lemma}\label{lemknice}
    Let $\sigma\in\mathbb{S}_3$ and $k\geq1$.
    If each algebra of the variety defined by an identity of associative type $\sigma$
    is $k$-nice,
    then $\sigma=(123)$ and the identity is of the first type,
    i.e., the variety is the variety of shift associative algebras.
\end{lemma}
\begin{proof}
Let $\mathfrak{V}$ be the variety defined by an identity of associative type $\sigma$ (passing to the opposite algebras, we may assume that the identity has the form given in the definition).
If the identity is of the second type, then it can be applied only to subwords of the form $(ab)c$;
hence, different right-normed monomials $x_{1}(x_{2}(\cdots(x_{k-1}x_{k})\cdots))$ and $x_{2}(x_{1}(\cdots(x_{k-1}x_{k})\cdots))$ are different elements of the free algebra of $\mathfrak{V}$, and $\mathfrak{V}$ contains algebras which are not $k$-nice.
Let the identity be of the first type, $(x_1x_2)x_3=x_{\sigma(1)}(x_{\sigma(2)}x_{\sigma(3)}).$
For $\sigma={\rm id}$, the variety $\mathfrak{V}$ contains the free associative algebra, which is not $k$-nice.
For $\sigma=(132)$ and $\sigma=(13)$, every commutative algebra satisfies the defining identity, and the free commutative (nonassociative) algebra is not $k$-nice.
For $\sigma=(23)$ (resp., $\sigma=(12)$), the defining identity preserves the leftmost (resp., the rightmost) generator of a monomial; hence, in the free algebra of $\mathfrak{V}$, the monomials
$x_{1}(x_{2}(\cdots(x_{k-1}x_{k})\cdots))$ and $x_{2}(x_{1}(\cdots(x_{k-1}x_{k})\cdots))$ (resp., $x_{1}(x_{2}(\cdots(x_{k-1}x_{k})\cdots))$ and $x_{1}(x_{2}(\cdots(x_{k}x_{k-1})\cdots))$)
are different, and $\mathfrak{V}$ contains algebras which are not $k$-nice.
Finally, for $\sigma=(123)$ the identity defines the variety of shift associative algebras.
\end{proof}

\begin{corollary}
Let $\mathcal{A}$ be a shift associative algebra. Then
$(\mathcal{A}, [\cdot,\cdot])$ is a nilpotent Lie algebra of class at most $4$,
i.e., it satisfies
\begin{center}
$\left[ u,\left[ v,\left[ x,\left[ y,z\right] \right] \right] \right] =0.$
\end{center}
If $\mathcal{A}$ is a cyclic associative algebra, then
$(\mathcal{A}, [\cdot,\cdot])$ is nilpotent of class at most $3$.
Moreover, if the multiplication of $(\mathcal{A}, [\cdot,\cdot])$ is zero, then $\mathcal{A}$ is a commutative
associative algebra.
\end{corollary}

\begin{corollary}\label{A2assoc}
  Let $\mathcal{A}$ be a shift associative algebra.
  Then $\mathcal{A}^2$ is associative and $\mathcal{A}^3$ is commutative associative. If $\mathcal{A}$ is cyclic associative,
  then $\mathcal{A}^2$ is commutative associative.
\end{corollary}

\begin{definition}[see, \cite{isz}]
    An algebra  $\mathcal{A}$ is called two-step associative if it satisfies
    \begin{center}
        $((x,y,z),w,t)=0.$
    \end{center}
\end{definition}

\begin{corollary}
Let $\mathcal{A}$ be a shift associative algebra. Then $\mathcal{A}$ is a two-step associative algebra.
\end{corollary}

\begin{definition}[see, \cite{KM24}]
        Let $p,q\in\mathcal{A}$. The algebra $(\mathcal{A}, *_{p,q})$ with the multiplication $x *_{p,q} y = (xp)y-(yq)x$ is called the $(p,q)$-mutation of $\mathcal{A}$.
\end{definition}

\begin{lemma}
Let $\mathcal{A}$ be a shift associative algebra. 
Then the mutation $(\mathcal{A}, *_{p,q})$  is a cyclic associative algebra.
\end{lemma}

\begin{proof}
By Theorem \ref{k-nice}, the product of any five elements does not depend on their association and order; denoting by $xyz\,(p-q)^2$ the corresponding linear combination of such products, we have
{\setlength{\tabcolsep}{2pt}\begin{longtable}{rclcl}
$(x  *_{p,q} y)  *_{p,q} z $&$ =$&$ \big((xp)y-(yq)x\big)  *_{p,q} z \ =$\\ 
& $=$&$(((xp)y)p)z - (zq)((xp)y) - (((yq)x)p)z + (zq)((yq)x) $&$ =$&$   xyz (p-q)^2$\\

$x  *_{p,q} (y  *_{p,q} z) $&$ =$&$   x*_{p,q} \big( (yp)z- (zq)y\big) \ =$\\ 
& $=$&$ (xp)((yp)z) -(((yp)z)q)x -(xp)( (zq)y) + (( (zq)y)q)x $&$ = $&$   xyz (p-q)^2$\\

$y  *_{p,q} (z  *_{p,q} x) $&$ =$&$ y *_{p,q} \big((zp)x-(xq)z\big) \ =$\\ 
& $=$&$(yp)((zp)x) -(((zp)x)q)y - (yp)((xq)z) + (((xq)z)q)y $&$ =$&$   xyz (p-q)^2.$\\
\end{longtable}}
Hence, we have our statement.

\end{proof}

Recall that the mutation of each perm algebra
(i.e., an associative algebra of the second associative type with $\sigma=(12)$) belongs to the variety of dual cyclic associative algebras \cite{KM24}.
A version of the present mutation with $p$ and $q$ from the ground field was considered in \cite[Section 6]{BBR}; we call it a scalar mutation.
It was proved there that
each scalar mutation of a shift associative algebra is
shift associative if and only if the algebra is anti-flexible.
The next proposition shows that algebras of another associative type are more stable under scalar mutations.

\begin{proposition}
    Each scalar mutation of an algebra of the first associative type with 
    $\sigma=(132)$ is an algebra of the first associative type with 
    $\sigma=(132).$ 
\end{proposition}
\begin{proof}
 Consider the multiplication $x *_{\alpha,\beta} y= \alpha xy+\beta yx,$ where $\alpha$ and $\beta$ belong to the ground field.
Then
\begin{longtable}{lcl}

$(x *_{\alpha,\beta}y)*_{\alpha,\beta} z$ &$=$&$
\alpha^2 (xy)z+\alpha\beta\big((yx)z+z(xy)\big)+\beta^2z(yx),$\\

$ z*_{\alpha,\beta} (x *_{\alpha,\beta}y)$ &$=$&$
\alpha^2 z(xy)+\alpha\beta \big(z(yx)+(xy)z\big)+\beta^2(yx)z.$\\
\end{longtable}
By the identity $(xy)z=z(xy)$, the right-hand sides coincide, which gives the statement.
\end{proof}

\begin{definition}[see, \cite{FK}]
        Let $p\in\mathcal{A}$. The algebra $(\mathcal{A}, *_p)$ with the multiplication
  \begin{center}
        $x *_{p} y = p(xy)-(px)y-x(py)$
  \end{center}
\end{definition}

\begin{corollary}
Let $\mathcal{A}$ be a shift associative algebra. 
Then the $p$-Kantor square $(\mathcal{A}, *_{p})$  is a cyclic associative algebra.
\end{corollary}

\begin{proof}
By Theorem \ref{k-nice}, we have 
\begin{longtable}{rclclclcl}
$(x  *_{p} y)  *_{p} z $&$ =$&
$x  *_{p} (y  *_{p} z) $&$ =$&
$y  *_{p} (z  *_{p} x) $&$ =$ &$   xyz p^2.$\\

\end{longtable}
Hence, we have our statement.

\end{proof}

\begin{corollary}
\label{ex=xe}
    Let $\mathcal{A}$ be a shift associative  algebra with an idempotent element $e$%
. Then $e  x=x  e$ for all $x \in\mathcal{A} $.
\end{corollary}
\begin{proof}
        Let $x\in \mathcal{A}$. Since $e  e=e$, by Theorem \ref{k-nice} we have 
    $x  e=x  \left( \left( e 
e\right)   \left( e  e\right) \right) =\left( \left(
e  e\right)   \left( e  e\right) \right)   x=ex$.  
\end{proof}

\begin{corollary}[Dubnov--Ivanov--Nagata--Higman theorem for shift associative algebras]\label{dinh}
Each shift associative nilalgebra of bounded index over a field of characteristic $0$ is nilpotent.
\end{corollary}

\begin{proof}
Let $\mathcal{A}$ be a shift associative nilalgebra of bounded index. By Corollary \ref{A2assoc}, $\mathcal{A}^2$ is an associative nilalgebra of bounded index; hence, by the Dubnov--Ivanov--Nagata--Higman theorem, there is $m$ such that the product of any $m$ elements of $\mathcal{A}^2$ is zero.
Let $N\geq\max\{2m,5\}$. By Theorem \ref{k-nice}, each product of $N$ elements of $\mathcal{A}$ coincides with a product of $\lfloor N/2\rfloor\geq m$ elements of $\mathcal{A}^2$ (we combine the factors into pairs and, if $N$ is odd, one pair into a triple); hence, it is zero. Thus, $\mathcal{A}$ is nilpotent.
\end{proof}
The assumption of bounded index cannot be omitted: the commutative associative algebra generated by $x_1, x_2, \ldots$ with the relations $x_i^2=0$ is a nilalgebra, but it is not nilpotent.
Any nilpotent algebra is solvable, and any solvable power-associative algebra is a nilalgebra of bounded index.  
\begin{corollary}
Any solvable shift associative algebra is nilpotent.
\end{corollary}
Thus, the concepts of nilpotent algebra, solvable algebra, and nilalgebra coincide for finite-dimensional shift associative algebras. Hence, there is a unique maximal nilpotent ideal $\mathcal{R}$ in any finite-dimensional shift associative algebra $\mathcal{A}$; we 
call $\mathcal{R}$ the  {radical} of $\mathcal{A}$ ($\mathcal{R}= \mathcal{R}ad(
\mathcal{A})$). 
\begin{corollary}\label{nilp}
Let $\mathcal{A}$ be a shift associative algebra. Then $\mathcal{A}$\ is
nilpotent if and only if $\mathcal{A}^{+}$ is nilpotent.
\end{corollary}

\begin{definition}[see, \cite{KP16}] Let $\mathcal A$ be an algebra, $n$ be a natural number $ \geq 2$, and $f$ be an arrangement of brackets on a product of length $n$. A linear mapping $D$ on $\mathcal A$ is called an $f$-Leibniz derivation of $\mathcal A$, if for any $a_1, \dots a_n \in \mathcal A$ we have
\begin{center}
$D([a_1,\ldots,a_n]_f) \ =\  \sum\limits_{i=1}^n[a_1,a_2,\ldots, D(a_i), \ldots a_n]_f.$
\end{center}
If $D$ is an $f$-Leibniz derivation for any arrangement $f$ of length $n$, then $D$ will be called a {Leibniz-derivation of $\mathcal A$ of order $n.$} 
\end{definition}

\begin{corollary}
Let $\mathcal{A}$ be a finite-dimensional shift associative algebra over a field of characteristic $0$. If $\mathcal{A}$\ admits
an invertible Leibniz-derivation, then $\mathcal{A}$ is nilpotent.
\end{corollary}
\begin{proof}
Let $D$ be an invertible Leibniz-derivation of $\mathcal{A}$. Then $D$ is also an
invertible Leibniz-derivation of $\mathcal{A}^{+}$ and 
therefore by \cite[Theorem 30]{KP16}, $\mathcal{A}^{+}$ is nilpotent.
Hence, by Corollary \ref{nilp}, $\mathcal{A}$ is nilpotent.
\end{proof}

\section{On the self-dual associative type operads}

Throughout this section, the ground field has characteristic $0$.
Following \cite{GK94}, we compute the Koszul dual operad $\SAs^!$ of the operad $\SAs$ governing the variety of shift associative algebras. In addition, we determine which operads of associative type are self-dual and Koszul.

\begin{theorem}
\label{SAs}
The operad $\SAs$ is self-dual and it is not Koszul.
\end{theorem}
\begin{proof}
First, let us fix a basis of the multilinear component of degree $3$ of the free shift associative algebra: it consists of all right-normed monomials, and the left-normed monomials are expressed as follows:
$$(ab)c=b(ca),\;(ba)c=a(cb),\;(ac)b=c(ba),$$
$$(ca)b=a(bc),\;(bc)a=c(ab),\;(cb)a=b(ac).$$
The defining identities of the operad $\SAs^!$ are obtained from the Lie-admissibility condition for $S\otimes U$, where $S$ is a shift associative algebra and $U$ is an $\SAs^!$-algebra.
Then
\begin{multline*}
[[a\otimes u,b\otimes v],c\otimes w]=(ab)c\otimes (uv)w-(ba)c\otimes (vu)w-
c(ab)\otimes w(uv)+c(ba)\otimes w(vu)=\\
b(ca)\otimes (uv)w+a(cb)\otimes -(vu)w+
c(ab)\otimes -w(uv)+c(ba)\otimes w(vu).
\end{multline*}
Similarly,
\begin{multline*}
[[b\otimes v,c\otimes w],a\otimes u]=(bc)a\otimes (vw)u-(cb)a\otimes (wv)u-
a(bc)\otimes u(vw)+a(cb)\otimes u(wv)=\\
c(ab)\otimes (vw)u+b(ac)\otimes -(wv)u+
a(bc)\otimes -u(vw)+a(cb)\otimes u(wv).
\end{multline*}
Further,
\begin{multline*}
[[c\otimes w,a\otimes u],b\otimes v]=(ca)b\otimes (wu)v-(ac)b\otimes (uw)v-
b(ca)\otimes v(wu)+b(ac)\otimes v(uw)=\\
a(bc)\otimes (wu)v+c(ba)\otimes -(uw)v+
b(ca)\otimes -v(wu)+b(ac)\otimes v(uw).
\end{multline*}
Therefore,
\begin{multline*}
[[a\otimes u,b\otimes v],c\otimes w]+[[b\otimes v,c\otimes w],a\otimes u]+[[c\otimes w,a\otimes u],b\otimes v]=\\
b(ca)\otimes\{(uv)w-v(wu)\}+a(cb)\otimes\{(vu)w-u(wv)\}+c(ab)\otimes\{(vw)u-w(uv)\}+\\
b(ac)\otimes\{(wv)u-v(uw)\}+a(bc)\otimes\{(wu)v-u(vw)\}+c(ba)\otimes\{(uw)v-w(vu)\}.
\end{multline*}
We obtain the identity
$$(uv)w-v(wu)=0,$$
which means that the operad $\SAs$ is self-dual.
Calculating the dimensions of the components of the operad $\SAs$ up to degree $5$ by means of the package of Dotsenko and Heijltjes\footnote{V. Dotsenko, W. Heijltjes, Gr\"obner bases for operads,
\url{http://irma.math.unistra.fr/~dotsenko/Operads.html} (2019).}, we obtain the beginning of the Hilbert series of $\SAs$ and $\SAs^{!}$:
$$H(t)=H^!(t)=-t+t^2-t^3+t^4/2-t^5/120+O(t^6).$$
Thus,
$$H(H^!(t))=t + 61t^5/60+O(t^6)\neq t.$$
By \cite{GK94}, the operad $\SAs$ is not Koszul.
\end{proof}

\begin{theorem}\label{self-d}
For the operads of the first associative type, we have the following:
\begin{longtable}{ | l | l| l | } 
  \hline
  operad type& dual operad &  Koszulity  \\ 
  \hline
  $(ab)c=a(bc)$ & self-dual & Koszul \\   \hline
  $(ab)c=b(ca)$  & self-dual & not Koszul  \\ \hline
  $(ab)c=c(ab)$ & self-dual & Koszul  \\ \hline
  $(ab)c=a(cb)$  & $(ab)c=-a(cb)$ & not Koszul \\ \hline
  $(ab)c=b(ac)$  & $(ab)c=-b(ac)$ & not Koszul \\ \hline
  $(ab)c=c(ba)$ &  $(ab)c=-c(ba)$ & Koszul  \\ \hline
\end{longtable}
 
\end{theorem}
\begin{proof}
For $\sigma={\rm id}$, the statement is the classical result on the associative operad.

Denote by $A_{(\sigma)}$ the operad governing the variety defined by the identity $(x_1x_2)x_3= x_{\sigma(1)}(x_{\sigma(2)}x_{\sigma(3)})$, and denote by $H_{(\sigma)}$ the beginning of its Hilbert series.

The operad $A_{(23)}$ is not self-dual, since a direct calculation shows that the operad $A_{(23)}^!$ is defined by the identity $(x_1x_2)x_3=-(23) x_1(x_2x_3)$. For these operads, we obtain
$$H_{(23)}(t)=-t+t^2-t^3+t^4/2-t^5/6$$
and
$$H_{(23)}^!(t)=-t+t^2-t^3+t^4/2-0 t^5.$$
Finally, we have
$$H_{(23)}(H_{(23)}^!(t))=t+7 t^5/6+O(t^6)\neq t.$$

The operad $A_{(12)}$ is not self-dual, since a direct calculation shows that the operad $A_{(12)}^!$ is defined by the identity $(x_1x_2)x_3=-(12) x_1(x_2x_3)$. Also, we have

$$H_{(12)}(t)=-t+t^2-t^3+t^4/2-t^5/6,$$
$$H_{(12)}^!(t)=-t+t^2-t^3+t^4/2-0 t^5,$$
and
$$H_{(12)}(H_{(12)}^!(t))=t+7 t^5/6+O(t^6)\neq t.$$

The statements on operad $A_{(123)}$ were proved in Theorem \ref{SAs}.

For the operads $A_{(132)}$ and $A_{(13)}$, the self-duality statements can be checked directly. To prove that these operads are Koszul, we consider their polarizations. The operad $A_{(132)}$ is isomorphic to the operad defined by the identities
$$[a,b]=-[b,a],\;\;\;a\circ b=b\circ a,$$
$$[[a,b],c]=0,\;\;\;[a\circ b,c]=0.$$
Similarly, the operad $A_{(13)}$ is isomorphic to the operad defined by the identities
$$[a,b]=-[b,a],\;\;\;a\circ b=b\circ a,$$
$$[a\circ b,c]=0,\;\;\;[a, b]\circ c=0.$$
The corresponding shuffle operads have monomial relations and, therefore, quadratic Gr\"obner bases. It follows that the operads $A_{(132)}$ and $A_{(13)}$ are Koszul \cite{Bremner-Dotsenko}.
\end{proof}

\begin{corollary}
        For any $\sigma\in\mathbb{S}_3$, the operad governing the variety defined by the identity of the first associative type $\sigma$
    is self-dual if and only if $\sigma $ is even.
\end{corollary}
\begin{remark}
For any $\sigma\in\mathbb{S}_3\setminus\{{\rm id}\}$, the operad governing the variety defined by the identity
$$(x_{1}x_2)x_3=(x_{\sigma(1)}x_{\sigma(2)})x_{\sigma(3)}$$ is not self-dual, since the degree-$3$ component of a self-dual operad generated by one binary operation without symmetries has dimension $6$.
\end{remark}

\begin{theorem}\label{table}
For the algebras of the first associative type, the following table lists identities satisfied by the anticommutator and by the commutator; in each filled cell, all identities of the corresponding operation follow from the listed ones (the question mark means that the answer is not known):
{\setlength{\tabcolsep}{2pt}\begin{longtable}{ | l | l| l | } 
  \hline
  operad type& anti-commutator &  commutator  \\ 
  \hline
  $(ab)c=a(bc)$ & ? & anti-com, Jacobi \\   \hline
  $(ab)c=b(ca)$  & com, (\ref{o1})--(\ref{o3}) & anti-com, Jacobi, $[[[[a,b],c],d],e]=0$  \\ \hline
  $(ab)c=c(ab)$ & com & anti-com, $[[a,b],c]=0$  \\ \hline
  $(ab)c=a(cb)$  & ? & ? \\ \hline
  $(ab)c=b(ac)$  & ? & ? \\ \hline
  $(ab)c=c(ba)$ &  com & anti-com  \\ \hline
\end{longtable}}
\noindent where com and anti-com are commutative and anti-commutative identities, respectively. 
\end{theorem}
\begin{proof}
For the operads $A_{(132)}$ and $A_{(13)}$, the statements follow from the description of their polarizations in the proof of Theorem \ref{self-d}: since these polarizations have monomial quadratic Gr\"obner bases, the anticommutator generates free commutative algebras, and the commutator generates free anticommutative algebras with the identity $[[a,b],c]=0$ (for $A_{(132)}$) and free anticommutative algebras (for $A_{(13)}$).
For the operad $\SAs$, the anticommutator satisfies (\ref{o1})--(\ref{o3}) and the commutator satisfies $[[[[a,b],c],d],e]=0$ by Lemma \ref{all0}.
A direct computation in the free shift associative algebra shows that
the left-normed commutators of degrees $3$ and $4$ are linearly independent modulo the Lie identities,
the anticommutator monomials of degree $3$ are linearly independent,
the anticommutator monomials of degree $4$ span a space of dimension $5$ (which is the dimension of the degree-$4$ multilinear component of the free commutative algebra defined by (\ref{o1})--(\ref{o3})),
and the identities (\ref{o1})--(\ref{o3}) imply that all multilinear anticommutator monomials of degree $5$ coincide.
Since all products of degree at least $5$ in a shift associative algebra do not depend on the association and order (Theorem \ref{k-nice}), the statement for $\SAs$ follows.
\end{proof}

\begin{theorem}\label{cycid}
Any cyclic associative algebra satisfies the following identities:
$$[a,b]=-[b,a],\;\;[[a,b],c]=0,$$
$$a\circ b=b\circ a,\;\;(a\circ b)\circ c=a\circ(b\circ c).$$
All identities satisfied by the commutator and all identities satisfied by the anticommutator of cyclic associative algebras follow from the identities above.
\end{theorem}
\begin{proof}
It follows from Theorems \ref{Jor-admiss} and \ref{k-nice} and the description of a basis of the free cyclic associative algebra given in Theorem \ref{teofreec} below.
\end{proof}

\begin{remark}
For any $\sigma\in \mathbb{S}_3$, 
the variety of algebras of the second associative type $\sigma$
defined by the identity $(x_1x_2)x_3=(x_{\sigma(1)}x_{\sigma(2)})x_{\sigma(3)}$ is not Lie-admissible.
\end{remark}

\noindent{\bf Open question.}
Complete the table  in Theorem \ref{table} and
construct an analogous table as in Theorem \ref{table} for algebras
of the second associative type.

\section{Free shift associative algebras}\label{secfree}

Let $X=\{x_1,x_2,\ldots\}$ be a countable set, and let $\SAs(X)$ be the free shift associative algebra generated by $X$. In this section, we construct a basis of $\SAs(X)$. We set
$$\mathcal{B}_1=\{x_i\},\;\;\;\mathcal{B}_2=\{x_ix_j\},\;\;\;\mathcal{B}_3=\{x_i(x_jx_k)\},$$
$$\mathcal{B}_4=\{x_i(x_j(x_kx_l))\;|\;(i,j)\leq(k,l)\},$$
where $\leq$ is the lexicographic order on pairs of indices,
and
$$\mathcal{B}_n=\{x_{i_1}\circ(x_{i_2}\circ(\cdots(x_{i_{n-1}}\circ x_{i_n})\cdots))\;|\;i_1\leq i_2\leq\cdots\leq i_n\},$$
where $x\circ y=1/2(xy+yx)$ and $n\geq 5$.

\begin{lemma}\label{all0}
Every shift associative algebra satisfies the following identities:
\begin{multline*}
    [a,[b,[c,[d,e]]]]=[a,[b,[c,d\circ e]]]=[a,[b,c\circ[d,e]]]=\\
    [a,b\circ[c,[d,e]]]=a\circ[b,[c,[d,e]]]=0,
\end{multline*}
\begin{multline*}
    [a,[b,c\circ(d\circ e)]]=[a,b\circ[c,d\circ e]]=a\circ[b,[c,d\circ e]]=[a,b\circ(c\circ[d,e])]=\\
    a\circ[b,c\circ[d,e]]=a\circ(b\circ[c,[d,e]])=0,
\end{multline*}
    $$a\circ(b\circ(c\circ[d,e]))=a\circ(b\circ[c,d\circ e])=a\circ[b,c\circ(d\circ e)]=[a,b\circ(c\circ(d\circ e))]=0,$$
\begin{equation}\label{o1}a\circ(b\circ(c\circ d))=b\circ(a\circ(c\circ d)),\end{equation}
\begin{equation}\label{o2}a\circ(b\circ(c\circ d))=c\circ(d\circ(a\circ b)),\end{equation}
\begin{equation}\label{o3}(a\circ b)\circ(c\circ d)+(a\circ c)\circ(b\circ d)+(a\circ d)\circ(b\circ c)=a\circ(b\circ(c\circ d))+a\circ(c\circ(b\circ d))+a\circ(d\circ(b\circ c)),\end{equation}
$$a\circ(b\circ(c\circ (d\circ e)))=a\circ(b\circ(d\circ (c\circ e))),$$
$$a(b(cd))=c(d(ab)).$$
\end{lemma}
\begin{proof}
The identities can be verified by computer algebra, for example, by the program Albert\footnote{Albert version 4.0M6; {\em https://web.osu.cz/~Zusmanovich/soft/albert/}}.
\end{proof}

\begin{theorem}\label{teofree}
The set $\bigcup_i\mathcal{B}_i$ is a basis of $\SAs(X)$.
\end{theorem}
\begin{proof}
First, let us show that every monomial of $\SAs(X)$ is a linear combination of elements of $\bigcup_i\mathcal{B}_i$. By the defining identity, every monomial is equal to a right-normed monomial; for monomials of degree at most $4$, it remains to apply the last identity of Lemma \ref{all0}. Starting from degree $5$, we observe that any monomial of $\SAs(X)$ can be written as a linear combination of
\begin{equation}\label{right-norm}
x_{i_1}\star(x_{i_2}\star(\cdots(x_{i_{n-1}}\star x_{i_n})\cdots)),    
\end{equation}
where $\star$ is $\circ$ or $[\cdot,\cdot]$.
Indeed, by the defining identity, every monomial is equal to a right-normed monomial. Since every monomial can be expressed in terms of the operations $\circ$ and $[\cdot,\cdot]$, i.e.,
$$xy=x\circ y+[x,y],$$
for every right-normed monomial, we have
$$x_{i_1}(x_{i_2}(\cdots(x_{i_{n-1}}x_{i_n})\cdots))=\sum_i x_{i_1}\star(x_{i_2}\star(\cdots(x_{i_{n-1}}\star x_{i_n})\cdots)),$$
due to the commutative and anti-commutative identities of the operations $\circ$ and $[\cdot,\cdot]$, respectively.

By Lemma \ref{all0}, every right normed monomial of the form (\ref{right-norm}), except
$$x_{i_1}\circ(x_{i_2}\circ(\cdots(x_{i_{n-1}}\circ x_{i_n})\cdots)),$$
is $0$. It remains to show that the generators $x_{i_1},\ldots,x_{i_n}$ can be ordered; this follows from the commutativity of $\circ$ and the identities of Lemma \ref{all0} containing only the operation $\circ$.

Finally, the elements of $\mathcal{B}_n$, $n\geq5$, are linearly independent, since the variety of commutative associative algebras is a subvariety of the variety of shift associative algebras; for $n\leq 4$, the linear independence of $\mathcal{B}_n$ is verified by a direct computation.

\end{proof}

  Let $\CAs(X)$ be the free cyclic associative algebra generated by $X$. We set
$$\mathcal{B}_1=\{x_i\},\;\;\;
\mathcal{B}_2=\{x_ix_j\},\;\;\;
\mathcal{B}_3=\{x_i(x_jx_k)\;|\; i < \min\{j,k\}\}\cup\{x_i(x_ix_k)\;|\; i\leq k\},$$
and
$$\mathcal{B}_n=\{x_{i_1}\circ(x_{i_2}\circ(\cdots(x_{i_{n-1}}\circ x_{i_n})\cdots))\;|\;i_1\leq i_2\leq\cdots\leq i_n\},$$
where $x\circ y=1/2(xy+yx)$ and $n \geq 4.$

Note that $\mathcal{B}_3$ consists of the monomials $x_i(x_jx_k)$ such that the word $x_ix_jx_k$ is the lexicographically smallest among its cyclic shifts.
The next theorem can be proved similarly to Theorem \ref{teofree}.

\begin{theorem}\label{teofreec}
The set $\bigcup_i\mathcal{B}_i$ is a basis of $\CAs(X)$.
\end{theorem}

\begin{remark}
The bases of $\SAs(X)$  and $\CAs(X)$  and direct calculations prove Theorem \ref{k-nice}.
\end{remark}

Similarly, the Koszul dual of the cyclic associative operad is the operad $\CAs^!$ defined by the identity\footnote{Recall that this identity appeared in the study of mutations of associative algebras of the second associative type with $\sigma=(12)$ in \cite{KM24},
and each flexible dual cyclic associative algebra is weakly associative (about weakly associative algebras, see \cite{ale} and references therein).}
$$(a,b,c)+(b,c,a)+(c,a,b)=0.$$
Since the operad $\SAs$ is self-dual, we obtain the following inclusions of varieties:
$$\CAs\subset\SAs \subset\CAs^!.$$

\noindent{\bf Open question.}
Construct a basis of the free algebra $\CAs^!(X)$. Consider this algebra under commutator and anti-commutator.

\section{Structure theory for shift associative  algebras}
 
Throughout this section, $\mathbb{F}$ is an algebraically closed
field of characteristic different from $2$, and all algebras are finite-dimensional over $\mathbb{F}$.
 
\begin{theorem}[Wedderburn–Malcev Theorem]
\label{sum}
Let $\mathcal{A}$ be a shift associative algebra. Then $\mathcal{A}=\mathcal{S}\oplus \mathcal{R}$ (as vector spaces), where $\mathcal{S}$ is a semisimple commutative
associative algebra and $\mathcal{R}$\ is the radical of $\mathcal{A}$ ($\mathcal{R}= \mathcal{R}ad(\mathcal{A})$).
\end{theorem}

\begin{proof}
Let $\mathcal{A}$ be a shift associative algebra. Then, by Theorem \ref{Jor-admiss}, $\mathcal{A}^{+}$ is
a Jordan algebra. Recall that $\mathcal{A}^{+}$ is the algebra with the same underlying vector space as $\mathcal{A}$
and the multiplication $x\circ y=\frac{1}{2}\left( x  y+y  x\right)$. By the Wedderburn principal theorem for Jordan algebras, we may write $\mathcal{A}=\mathcal{S}\oplus \mathcal{R}$ where $\left( 
\mathcal{S},\circ \right) $ is a semisimple Jordan algebra and $\left( 
\mathcal{R},\circ \right) $\ is the radical of $\mathcal{A}^{+}$. Moreover, we have:%
\begin{longtable}{rclrcl}
$\mathcal{S}  \mathcal{S} $&$= $&$\mathcal{S}\circ \mathcal{S}+\left[ \mathcal{
S},\mathcal{S}\right],$ &  $\mathcal{S}  \mathcal{R}$&$ =$&$\mathcal{S}\circ \mathcal{R}+\left[ \mathcal{S},\mathcal{R}\right],$\\   
$\mathcal{R}  \mathcal{R} $&$=$&$\mathcal{R}\circ \mathcal{R}+\left[ \mathcal{R},\mathcal{R}\right],$&$\mathcal{R}  \mathcal{S} $&$=$&$\mathcal{R}\circ \mathcal{S}+\left[ \mathcal{R},\mathcal{S}\right].$
\end{longtable}
Since $\left( 
\mathcal{S},\circ \right) $ is a semisimple Jordan algebra, $\mathcal{ S}\circ\mathcal{ S}=\mathcal{S}$. Then, by Theorem \ref{k-nice}, we have%
    \begin{longtable}{rclrcl}
      $\left[ \mathcal{S},\mathcal{S}\right] $&$ =$&$\left[ \mathcal{S},\left( \mathcal{ S}\circ\mathcal{ S}\right) \circ \left( \mathcal{ S}\circ\mathcal{ S}\right) \right] =0,$   &  
      $\left[ \mathcal{S},\mathcal{R}\right] $&$ =$&$\left[ \left( \mathcal{ S}\circ\mathcal{ S}\right) \circ \left( \mathcal{ S}\circ\mathcal{ S}\right) ,\mathcal{R}\right]
=0.$   
    \end{longtable}
 
Hence, we have the following relations:
\begin{equation*}
    \begin{array}{ccc}
     \mathcal{S}   \mathcal{S}= \mathcal{S}\circ  \mathcal{S}= \mathcal{S},    & \mathcal{S}  \mathcal{R}=\mathcal{S}\circ \mathcal{R}\subseteq \mathcal{R},
         & \mathcal{R}  \mathcal{S}=\mathcal{R}\circ \mathcal{S}\subseteq \mathcal{R}.
    \end{array}
\end{equation*}
Then $\mathcal{S}$ is a semisimple Jordan algebra.
Any semisimple Jordan algebra has a unit element. So, by Lemma \ref{unital}, $\mathcal{S}$ is a semisimple commutative associative algebra. That is, $\mathcal{S}=%
\underset{i=1}{\overset{s}{\oplus }}\mathbb{F}e_{i}$ with $e_{i} 
e_{i}=e_{i}$ and $e_{i}  e_{j}=0$ if $i\neq j$.  Furthermore, we have
\begin{equation*}
\mathcal{R}  \mathcal{R}=\mathcal{R}\circ \mathcal{R}+\left[ \mathcal{R},%
\mathcal{R}\right] \subseteq \mathcal{R}+\left[ \mathcal{R},\mathcal{R}%
\right].
\end{equation*}%
We claim that $\left[ \mathcal{R},\mathcal{R}\right] \subseteq \mathcal{R}$.
To see this, let $x,y\in \mathcal{R}$ such that $\left[ x,y\right] =n+%
\overset{s}{\underset{i=1}{\sum }}\alpha _{i}e_{i}$ where $n\in \mathcal{R}$ and 
$\overset{s}{%
\underset{i=1}{\sum }}\alpha _{i}e_{i}\in \mathcal{S}$. Note that if $a,b,c,d\in \mathcal{A\textit{}}$, then by $\left(\ref{swap}\right) $, we
have 
\begin{equation*}
\left[ a,b\right]   \left( c\circ d\right) =(a  b)  \left(
c  d\right) +(a  b)  \left( d  c\right) -(b  a) 
\left( c  d\right) -(b  a)  \left( d  c\right) =0.
\end{equation*}
Thus, for each $j\in \{1,2,\ldots ,s\}$, we have%
\begin{equation*}
0=\left[ x,y\right]   \left( e_{j}\circ e_{j}\right)=\left[ x,y\right]   \left( e_{j}  e_{j}\right) =\left[ x,y\right]   e_{j}=\left( 
n+\overset{s}{\underset{i=1}{\sum }}\alpha _{i}e_{i}\right)   e_{j}=n  e_{j}+\alpha _{j}e_{j}.
\end{equation*}%
Since $\mathcal{R}  \mathcal{S}\subseteq \mathcal{R}$, we have $n  e_{j}\in 
\mathcal{R}$ and therefore $n  e_{j}=\alpha _{j}e_{j}=0$. So $\alpha _{j}=0$ for $j=1,2,\ldots ,s$. Thus $\left[ \mathcal{R},%
\mathcal{R}\right] \subseteq \mathcal{R}$.
\end{proof}
 
The following result follows from the proof of Theorem \ref{sum}.
\begin{corollary}\label{radeq}
  Let $\mathcal{A}$ be a shift associative algebra. Then  $\mathcal{R}ad(\mathcal{A})=\mathcal{R}ad(\mathcal{A^+})$.
\end{corollary}

Let $\left( \mathcal{A},\circ \right) $ be a Jordan algebra with idempotent
element $e\neq 0$. Then the Peirce decomposition of $\mathcal{A}$ relative to the
idempotent $e$ is 
\begin{equation*}
\mathcal{A}=\mathcal{A}_{0}\oplus \mathcal{A}_{\frac{1}{2}}\oplus \mathcal{A}%
_{1}
\end{equation*}%
where $\mathcal{A}_{i}=\mathcal{A}_{i}(e):=\{x\in \mathcal{A}:x\circ e=ix\}$%
, with $i\in \{0,\frac{1}{2},1\}$. The multiplication table for the
Peirce decomposition is: 
\begin{equation*}
\begin{array}{ccc}
\mathcal{A}_{0}\circ \mathcal{A}_{0}\subseteq \mathcal{A}_{0}, & \mathcal{A}%
_{\frac{1}{2}}\circ \mathcal{A}_{\frac{1}{2}}\subseteq \mathcal{A}_{0}\oplus 
\mathcal{A}_{1}, & \mathcal{A}_{1}\circ \mathcal{A}_{1}\subseteq \mathcal{A}%
_{1}, \\ 
\mathcal{A}_{0}\circ \mathcal{A}_{\frac{1}{2}}\subseteq \mathcal{A}_{\frac{1%
}{2}}, & \mathcal{A}_{0}\circ \mathcal{A}_{1}=0, & \mathcal{A}_{\frac{1}{2}%
}\circ \mathcal{A}_{1}\subseteq \mathcal{A}_{\frac{1}{2}}.%
\end{array}
\end{equation*}

\begin{theorem}
\label{decom}
Let $\mathcal{A}$ be a shift associative  algebra with an idempotent element $e\neq 0$%
. Then the Peirce decomposition of $\mathcal{A}$ relative to the idempotent $%
e$ is: 
\begin{equation*}
\mathcal{A}=\mathcal{A}_{0}\oplus \mathcal{A}_{1}
\end{equation*}%
where $\mathcal{A}_{i}=\mathcal{A}_{i}(e):=\{x\in \mathcal{A}:x 
e+e  x=2ix\}$, with $i\in \{0,1\}$. The multiplication table for
the Peirce decomposition is:%
\begin{equation*}
\begin{array}{cccc}
\mathcal{A}_{0}  \mathcal{A}_{0}\subseteq \mathcal{A}_{0}, & \mathcal{A}%
_{0}  \mathcal{A}_{1}=0, & \mathcal{A}_{1}  \mathcal{A}_{0}=0, & 
\mathcal{A}_{1}  \mathcal{A}_{1}\subseteq \mathcal{A}_{1}.%
\end{array}%
\end{equation*}%
That is, $\mathcal{A}_{0},\mathcal{A}_{1}$ are ideals of $\mathcal{A}$.
\end{theorem}

\begin{proof}
Let $\mathcal{A}$ be a shift associative  algebra with an idempotent element $e$%
. It is clear that an idempotent of $\mathcal{A}$ is also an idempotent of $%
\mathcal{A}^{+}$. Since $\mathcal{A%
}^{+}$ is the Jordan algebra with the same underlying vector space as $\mathcal{A}$
but with $x\circ y=\frac{1}{2}\left( x  y+y  x\right) $ as the
multiplication, $\mathcal{A}=\mathcal{A}_{0}\oplus \mathcal{A}_{\frac{1}{2}%
}\oplus \mathcal{A}_{1}$ where $\mathcal{A}_{i}=\mathcal{A}_{i}(e):=\{x\in 
\mathcal{A}:x\circ e=\frac{1}{2}\left( x  e+e  x\right) =ix\}$, with 
$i\in \{0,\frac{1}{2},1\}$.

\medskip
First we claim that $\mathcal{A}_{\frac{1}{2}}=0$.
To see this, let $x\in \mathcal{A}_{\frac{1}{2}}$. Then $x  e+e  x=x$%
. By Corollary \ref{ex=xe},
we have  $e  x=x  e$ and therefore $x=2e  x$. Multiplying the
last identity by $e$ on the left, we get $e  x=2e  \left( e 
x\right) $. But, by $\left( \ref{nearly_eq}\right) $, we have $e  x=\left( e  e\right)  
x=e  \left( x  e\right) =e  \left( e  x\right) $ (also, by Theorem \ref{k-nice}, $e  x=\left( \left( e  e\right)  \left( e  e\right) \right)  
x=\left( \left( e  e\right)   e\right)   \left( e  x\right)
=e  \left( e  x\right) $). Thus $%
e  x=0$ and so $x=0$.

\medskip
Next, let $x\in \mathcal{A}_{0}$ and $y\in \mathcal{A}$. 
Then $e  x+x  e=0$. So, by Theorem \ref{k-nice}, we have
\begin{longtable}{lclclcl}
$e  \left( x  y\right) +\left( x  y\right)   e$&$ =$&$\left(
\left( e  e\right)   e\right)   \left( x  y\right) +\left(
x  y\right)   \left( \left( e  e\right)   e\right)
$&$=$&$\left( \left( \left( e  x+x  e\right)   e\right)  
e\right)   y 
$&$=$&$0,$ \\
$e  \left( y  x\right) +\left( y  x\right)   e $&$=$&$\left(
\left( e  e\right)   e\right)   \left( y  x\right) +\left(
y  x\right)   \left( \left( e  e\right)   e\right) 
$&$=$&$\left( \left( \left( e  x+x  e\right)   e\right)  
e\right)   y 
$&$=$&$0.$
\end{longtable}%
It follows that $\mathcal{A}_{0}  \mathcal{A}+\mathcal{A}  \mathcal{A}_{0}\subseteq \mathcal{A}_{0}$. This proves that $\mathcal{A}_{0}$ is an ideal of $\mathcal{A}$. 

\medskip
Finally, let $x\in \mathcal{A}_{1}$ and $y\in \mathcal{A}$. Then $e 
x+x  e=2x$. Hence, by Theorem \ref{k-nice}, we have
\begin{longtable}{lclclclc}
$e  \left( x  y\right) +\left( x  y\right)   e $&$=$&$\left(
\left( e  e\right)   \left( e  e\right) \right)   \left(
x  y\right) +\left( x  y\right)   \left( \left( e  e\right)
  \left( e  e\right) \right)  $&$
=$\\
&$=$&$\left( \left( \left( e  e\right)   \left( e  e\right) \right)
  x\right)   y+\left( x  \left( \left( e  e\right)  
\left( e  e\right) \right) \right)   y $&$=$&$\left( e  x+x  e\right)   y $&$=$&$ 2x  y$
\end{longtable}%
and
\begin{longtable}{lclclclc}
$e  \left( y  x\right) +\left( y  x\right)   e $&$=$&$\left(
\left( e  e\right)   \left( e  e\right) \right)   \left(
y  x\right) +\left( y  x\right)   \left( \left( e  e\right)
  \left( e  e\right) \right)  $&$=$\\
  &$=$&$y\left( \left( \left( e  e\right)   \left( e  e\right) \right)
  x\right)   +y\left( x  \left( \left( e  e\right)  
\left( e  e\right) \right) \right)   $&$=$&$y\left( e  x+x  e\right)    $&$=$&$2yx.$
\end{longtable}%

Thus $\mathcal{A}_{1}  \mathcal{A}+\mathcal{A}  \mathcal{A}_{1}\subseteq \mathcal{A}_{1}$, so $\mathcal{A}_{1}$ is an ideal of $\mathcal{A}$. Since $\mathcal{A}_{0}\cap \mathcal{A}_{1}=0$, we have $\mathcal{A}_{0}  \mathcal{A}_{1}=\mathcal{A}_{1}  \mathcal{A}_{0}=0$.
\end{proof}

\begin{definition}\label{decomposable}\rm
An algebra $\mathcal A$ is  {decomposable} if there are non-zero ideals $I_1$ and $I_2$ such that $\mathcal A = I_1 \oplus I_2$. Otherwise, it is  {indecomposable}.
\end{definition}

\begin{lemma}
  \label{unique idempoten}
Let $\mathcal{A}$ be a non-nilpotent shift associative algebra. If $\mathcal{A}$ is indecomposable, then $\mathcal{A}$ is unital and has a unique idempotent, that is, $e=1_{\mathcal{A}}$.  
\end{lemma}

\begin{proof}
By Corollary \ref{idem} there exists a non-zero idempotent $e\in \mathcal{A}$ and, by
Theorem \ref{decom}, $\mathcal{A}=\mathcal{A}_0 \oplus \mathcal{A}_1$ is a direct sum of ideals of $\mathcal{A}$. Since $\mathcal{A}$ is indecomposable and $e \in \mathcal{A}_1 \neq 0$, it follows that $\mathcal{A}_0=0$ and $\mathcal{A}=\mathcal{A}_1$. Now, let $x\in \mathcal  A=\mathcal 
 A_{1}$. Then $e  x+x  e =2x$. By Corollary \ref{ex=xe}, we have  $e  x=x  e$ and therefore $e  x=x=x   e$. So $e=1_{\mathcal{A}}$. Suppose that $e'$ is an idempotent of $\mathcal{A}$.
Then $\mathcal{A}_1(e=1_{\mathcal{A}}) = \mathcal{A} = \mathcal{A}_1(e')$ and so $e'=1_{\mathcal{A}}$.
\end{proof}

\begin{proposition}
\label{simple=K}
Let $\mathcal{A}$ be a shift associative algebra. If $\mathcal{A}$ is simple then $\mathcal{A} \cong \mathbb{F}$.
\end{proposition}
\begin{proof}
By Lemma \ref{unique idempoten}, $\mathcal{A}$ is a unital algebra. Moreover, by Lemma \ref{unital}, $\mathcal{A}$ is a simple commutative associative algebra and so $\mathcal{A} \cong \mathbb{F}$.
\end{proof}

\begin{proposition}
\label{class}
Let $\mathcal{A}$ be a shift associative algebra. If $\mathcal{A}$
is indecomposable, then either:
\begin{itemize}
\item[i.] $\mathcal{A}$ is an indecomposable nilpotent
shift associative  algebra or

\item[ii.] $\mathcal{A}$ is the unital hull of a nilpotent commutative associative algebra $B$ (i.e. 
$\mathcal{A} = \widehat{B} = \mathbb{F}1_{\mathcal{A}}\oplus B$, so $\mathcal{A}$ is a local unital commutative associative algebra).
\end{itemize}
\end{proposition}

\begin{proof}
By Theorem \ref{sum}, $\mathcal{A} = \mathcal{S} \oplus \mathcal{R}ad(
\mathcal{A})$. If $\mathcal  S=0$ then $\mathcal{A} = \mathcal{R}ad(
\mathcal{A})$ is nilpotent. Otherwise, $\mathcal{A}$ is not
nilpotent. By Lemmas \ref{unique idempoten} and \ref{unital}, $\mathcal{A}$ is a unital commutative associative algebra with a unique
idempotent $e=1_{\mathcal{A}}$. If $\mathcal{R}ad( \mathcal{A})$ is
non-associative, then so is $\mathcal{A}$. Thus, $\mathcal{R}ad( 
\mathcal{A})$ is a nilpotent commutative associative algebra. By
Proposition \ref{simple=K} and since $\mathcal{A}$ has a unique idempotent $%
e=1_{\mathcal{A}}$, $\mathcal{S}=\mathbb{F}1_{\mathcal{A}}$ and hence $\mathcal{A}=\mathbb{F}1_{\mathcal{A}}\oplus 
\mathcal{R}ad( \mathcal{A})$.
\end{proof}

 By Definition \ref{decomposable} and Proposition \ref{class}, the classification of all shift associative algebras can be obtained from the classification of nilpotent shift associative algebras (one can adapt the Skjelbred--Sund method, used to classify nilpotent Lie algebras, to classify nilpotent shift associative algebras). Another approach to the classification of shift associative algebras is given in the next section.    

\section{Low dimensional shift associative  algebras}

Throughout this section, the ground field is the field of complex numbers.

\subsection{The algebraic classification of  shift associative  algebras}
The aim of this subsection is to improve \cite[Theorem 3.2]{Sergei} and to give a complete classification of shift associative algebras of dimension at most four.

\medskip
We first describe, up to
isomorphism, all shift associative algebras with a given nilpotent Lie algebra
as the associated (commutator) Lie algebra.
\medskip
Let $\left( \mathcal{L},\left[ -,-\right] \right) $ be a nilpotent Lie algebra. Define 
$\mathrm{Z}_{\mathrm{SA}}^{2}\left( \mathcal{L},\mathcal{L}\right) $ to be
the set of all symmetric bilinear maps $\theta :\mathcal{L}\times \mathcal{L}%
\longrightarrow \mathcal{L}$ such that:%
\begin{flushleft}
$\theta \left( x,\theta \left( y,z\right) \right) +\left[ x,\theta \left(
y,z\right) \right] +\theta \left( x,\left[ y,z\right] \right) +\left[ x,%
\left[ y,z\right] \right] \ =$\end{flushleft}
\begin{flushright}$= \ \theta \left( \theta \left( z,x\right) ,y\right) +%
\left[ \theta \left( z,x\right) ,y\right] +\theta \left( \left[ z,x\right]
,y\right) +\left[ \left[ z,x\right] ,y\right],$\end{flushright} 
for all $x,y,z$ in $\mathcal{L}$.

\begin{lemma}
Let $\mathcal{L}_{1},\mathcal{L}_{2}$ be two isomorphic nilpotent Lie algebras.
Then $\mathrm{Z}_{\mathrm{SA}%
}^{2}\left( \mathcal{L}_{1},\mathcal{L}_{1}\right) \neq \varnothing $ if and only if $\mathrm{Z}_{%
\mathrm{SA}}^{2}\left( \mathcal{L}_{2},\mathcal{L}_{2}\right)\neq \varnothing $. Moreover, there is a bijective correspondence between $\mathrm{Z}_{\mathrm{SA}%
}^{2}\left( \mathcal{L}_{1},\mathcal{L}_{1}\right) $\ and $\mathrm{Z}_{%
\mathrm{SA}}^{2}\left( \mathcal{L}_{2},\mathcal{L}_{2}\right) $.
\end{lemma}

\begin{proof}
Consider an isomorphism $\phi :\mathcal{L}_{2}\longrightarrow \mathcal{L}_{1}
$. We can prove $\theta \in \mathrm{Z}_{\mathrm{SA}}^{2}\left( \mathcal{L}%
_{1},\mathcal{L}_{1}\right) $ if and only if $\theta \ast \phi \in \mathrm{Z}%
_{\mathrm{SA}}^{2}\left( \mathcal{L}_{2},\mathcal{L}_{2}\right) $, where $%
\theta \ast \phi :\mathcal{L}_{2}\mathcal{\times \mathcal{L}}_{2}\mathcal{%
\longrightarrow L}_{2}$ is defined by $\left( \theta \ast \phi \right)
\left( x,y\right) =\phi ^{-1}\left( \theta \left( \phi \left( x\right) ,\phi
\left( y\right) \right) \right) $, for $x,y\in \mathcal{L}_{2}$. This
defines a bijective map $\tau :\mathrm{Z}_{\mathrm{SA}}^{2}\left( \mathcal{L}%
_{1},\mathcal{L}_{1}\right) \longrightarrow \mathrm{Z}_{\mathrm{SA}%
}^{2}\left( \mathcal{L}_{2},\mathcal{L}_{2}\right) $ by $\tau \left( \theta
\right) :=\theta \ast \phi $.
\end{proof}

If $\mathrm{Z}_{\mathrm{SA}}^{2}\left( \mathcal{L},\mathcal{L}%
\right) = \mathcal{\varnothing }$, then $\mathcal{L}$ does not have shift associative
structures, i.e., there is no shift associative  algebra $\mathcal{A}$ such that  $ \mathcal{A}^{-}\cong \mathcal{L}$. On the other hand in case of $\mathrm{Z}_{\mathrm{SA}}^{2}\left( \mathcal{L},\mathcal{L}%
\right) \neq \mathcal{\varnothing }$, observe that, for $\theta \in \mathrm{Z}_{\mathrm{SA}}^{2}\left( \mathcal{L},\mathcal{L}\right) $, if we define a
multiplication $ \cdot _{\theta }$ on $\mathcal{L}$ by $x  \cdot_{\theta }y=\theta
\left( x,y\right) +\left[ x,y\right] $ for all $x,y$ in $\mathcal{L}$, then $%
\left( \mathcal{L},  \cdot_{\theta }\right) $ is a shift associative
algebra. Conversely, if $\left( \mathcal{L}, \cdot \right) $ is a shift 
associative algebra, then there exists $\theta \in \mathrm{Z}_{\mathrm{SA}%
}^{2}\left( \mathcal{L},\mathcal{L}\right) $ such that $\left( \mathcal{L}%
, \cdot _{\theta }\right) \cong \left( \mathcal{L}, \cdot \right) $. To see
this, consider the symmetric bilinear map $\theta :\mathcal{L}\times \mathcal{L}%
\longrightarrow \mathcal{L}$ defined by $\theta \left( x,y\right) =x  y-%
\left[ x,y\right] $ for all $x,y$ in $\mathcal{L}$. Then $\theta \in \mathrm{%
Z}_{\mathrm{SA}}^{2}\left( \mathcal{L},\mathcal{L}\right) $ and $\left( 
\mathcal{L},  \cdot_{\theta }\right) =\left( \mathcal{L},  \cdot\right) $.

\bigskip

Now, let $\left( \mathcal{L},\left[ -,-\right] \right) $ be a nilpotent Lie algebra
such that $\mathrm{Z}_{\mathrm{SA}}^{2}\left( \mathcal{L},\mathcal{L}\right)
\neq \mathcal{\varnothing }$ and $\text{Aut}\left( \mathcal{L}\right) $ be
the automorphism group of $\mathcal{L}$. Then $\text{%
Aut}\left( \mathcal{L}\right) $ acts on $\mathrm{Z}_{\mathrm{SA}}^{2}\left( 
\mathcal{L},\mathcal{L}\right) $ by 
\[
\left( \theta \ast \phi \right) \left( x,y\right) =\phi ^{-1}\left( \theta
\left( \phi \left( x\right) ,\phi \left( y\right) \right) \right) , 
\]%
for $\phi \in \text{Aut}\left( \mathcal{L}\right) $, and $\theta \in \mathrm{%
Z}_{\mathrm{SA}}^{2}\left( \mathcal{L},\mathcal{L}\right) $.

\begin{lemma}
Let $\left( \mathcal{L},\left[ -,-\right] \right) $ be a nilpotent Lie algebra and $%
\mathrm{Z}_{\mathrm{SA}}^{2}\left( \mathcal{L},\mathcal{L}\right) \neq 
\mathcal{\varnothing }$. If $\theta ,\vartheta \in \mathrm{Z}_{\mathrm{SA}%
}^{2}\left( \mathcal{L},\mathcal{L}\right) $, then $\left( \mathcal{L} 
_{\theta }\right) $ and $\left( \mathcal{L}  _{\vartheta }\right) $ are
isomorphic if and only if there is a $\phi \in \text{Aut}\left( \mathcal{L}%
\right) $ with $\theta \ast \phi =\vartheta $.
\end{lemma}

\begin{proof}
If $\theta \ast \phi =\vartheta $, then $\phi :\left( \mathcal{L}, 
\cdot_{\vartheta }\right) \longrightarrow $ $\left( \mathcal{L},  \cdot_{\theta
}\right) $ is an isomorphism since $\phi \left( \vartheta \left( x,y\right)
\right) =\theta \left( \phi \left( x\right) ,\phi \left( y\right) \right) $.
On the other hand, if $\phi :\left( \mathcal{L},  \cdot_{\vartheta }\right)
\longrightarrow $ $\left( \mathcal{L}, \cdot _{\theta }\right) $ is an
isomorphism of shift associative  algebras, then $\phi \in \text{Aut}\left(
  \mathcal{L},\left[ -,-\right]   \right) $ and $\phi \left(
x  \cdot_{\vartheta }y\right) =\phi \left( x\right)  \cdot _{\theta }\phi
\left( y\right) $. Hence
\begin{center}
    $\vartheta \left( x,y\right) +\left[ x,y\right]
=\phi ^{-1}\left( \theta \left( \phi \left( x\right) ,\phi \left( y\right)
\right) +\left[ \phi \left( x\right) ,\phi \left( y\right) \right] \right)
=\left( \theta \ast \phi \right) \left( x,y\right) +\left[ x,y\right] $ 
\end{center} and
therefore $\theta \ast \phi =\vartheta $.
\end{proof}

Hence, we have a procedure to classify the shift associative  algebras that
have a given nilpotent Lie algebra as associated Lie algebra. It runs as follows:

\begin{enumerate}
\item Given a Lie algebra $\left( \mathcal{L},\left[ -,-\right] \right) $,
compute $\mathrm{Z}_{\mathrm{SA}}^{2}\left( \mathcal{L},\mathcal{L}\right) $.

\item If $\mathrm{Z}_{\mathrm{SA}}^{2}\left( \mathcal{L},\mathcal{L}\right) =%
\mathcal{\varnothing }$, then $\mathcal{L}$ does not have shift  associative
structures.

\item If $\mathrm{Z}_{\mathrm{SA}}^{2}\left( \mathcal{L},\mathcal{L}\right)
\neq \mathcal{\varnothing }$, then find the orbits of $\text{Aut}\left( 
\mathcal{L}\right) $ on $\mathrm{Z}_{\mathrm{SA}}^{2}\left( \mathcal{L},%
\mathcal{L}\right) $. Choose a representative $\theta $ from each orbit and
then construct the shift associative algebra $\left( \mathcal{L}, 
\cdot_{\theta }\right) $.
\end{enumerate}

\bigskip 

Let us introduce the following notations. Let $e_{1},e_{2},\ldots ,e_{n}$ be
a fixed basis of a nilpotent Lie algebra $\left( \mathcal{L},\left[ -,-%
\right] \right) $. Define $\mathrm{Sym}^{2}\left( \mathcal{L},\mathbb{C}%
\right) $ to be the space of all symmetric bilinear forms on $\mathcal{L}$.
Then  $\mathrm{Sym}^{2}\left( \mathcal{L},\mathbb{C}\right) =\left\langle
\Delta _{i,j}:1\leq i\leq j\leq n\right\rangle $ where $\Delta _{i,j}$ is
the symmetric bilinear form $\Delta _{i,j}:\mathcal{L}\times \mathcal{L}%
\longrightarrow \mathbb{C}$\ defined by%
\[
\Delta _{i,j}\left( e_{l},e_{m}\right) :=\left\{ 
\begin{tabular}{ll}
$1,$ & if $\left\{ i,j\right\} =\left\{ l,m\right\} ,$ \\ 
$0,$ & otherwise.%
\end{tabular}%
\right. 
\]%
Now, if $\theta \in \mathrm{Z}_{\mathrm{SA}}^{2}\left( \mathcal{L},\mathcal{L%
}\right) $, then $\theta $ can be uniquely written as $\theta \left(
x,y\right) =\underset{i=1}{\overset{n}{\sum }}B_{i}\left( x,y\right) e_{i}$
where $B_{1},B_{2},\ldots ,B_{n}$ is a sequence of symmetric bilinear forms
on $\mathcal{L}$. Also, we may write $\theta =\left( B_{1},B_{2},\ldots
,B_{n}\right).$ Let $\phi ^{-1}\in \text{Aut}\left( \mathcal{L}\right) $
be given by the matrix $\left( b_{ij}\right) $. If $\left( \theta \ast \phi
\right) \left( x,y\right) =\underset{i=1}{\overset{n}{\sum }}B_{i}^{\prime
}\left( x,y\right) e_{i}$, then $B_{i}^{\prime }=\underset{j=1}{\overset{n}{%
\sum }}b_{ij}\phi ^{t}B_{j}\phi $.

\noindent {\bf Classification in dimension 2}

\begin{lemma}\label{2-dim asscom}
Any nontrivial complex $2$-dimensional shift associative  algebra is isomorphic to one of the
following commutative associative algebras:

\begin{longtable}{cclll}
 $\mathcal{A}_{01}$&$:$&$e_{1} \circ e_{1}=e_{1}$&$e_{2} \circ e_{2}=e_{2}$\\

 $\mathcal{A}_{02}$&$:$&$e_{1} \circ e_{1}=e_{1} $&$e_{1}  \circ
e_{2}=e_{2} $ \\

 $\mathcal{A}_{03}$&$:$&$e_{1} \circ e_{1}=e_{1}$\\

 $\mathcal{A}_{04}$&$:$&$e_{1}\circ e_{1}=e_{2}$
\end{longtable}
\end{lemma}

\begin{proof}
Let $\mathcal{A}$ be a $2$-dimensional shift associative algebra. Then $(\mathcal{A}, [\cdot, \cdot])$ is a nilpotent Lie algebra and therefore it has zero multiplication. Hence $%
\mathcal{A}$\ is a commutative associative algebra.
\end{proof}

\begin{definition}
     Let $\Omega$ be a variety of algebras. 
 We say that an algebra $(\bf A, \cdot, \ast)$ is a compatible $\Omega$-algebra, 
 if and only if $(\bf A, \cdot),$ $(\bf A, \ast)$
and $(\bf A, \cdot + \ast)$ are $\Omega$-algebras.
In particular, 
$(\bf A, \cdot, \ast)$ is a compatible shift associative algebra, if it satisfies
\begin{longtable}{rcl}
$\left( x\cdot   y\right) \cdot   z$&$=$&$ y\cdot  \left(z\cdot x\right),$\\
$\left( x\ast   y\right) \ast   z$&$=$&$ y\ast  \left(z\ast x\right),$\\
$\left( x\ast   y\right) \cdot   z+\left( x\cdot   y\right) \ast   z$&$=$&$y\ast  \left(z\cdot x\right)+ y\cdot  \left(z\ast x\right).$\\
 \end{longtable}
\end{definition}

\begin{corollary}
    Let  $(\bf A, \cdot, \ast)$ be a complex $2$-dimensional
    compatible shift associative algebra. Then it is
    a compatible commutative associative algebra,
    and it is one of the algebras given in \cite[Theorem 8]{akm}.
\end{corollary}

\noindent {\bf Classification in dimension 3}

It is known that there is only one nontrivial  $3$-dimensional nilpotent Lie algebra: 
\begin{longtable}{clll}

 $\mathcal{L}_{1}$ &$:$&$\left[ e_{1},e_{2}\right] =e_{3}$
\end{longtable}

\begin{lemma}
\label{3-dim asscom}
Let $\mathcal{A}$ be a nontrivial  $3$-dimensional commutative associative algebra. Then $\mathcal{A}= \mathcal{A}_a \oplus \mathcal{A}_b,$
where $\mathcal{A}_a$ is given in Lemma \ref{2-dim asscom} and $\mathcal{A}_b^2=0,$ or it 
is isomorphic to one of the following algebras:

\begin{longtable}{cclllll}
 
$\mathcal{A}_{05}$&$:$&$e_{1} \circ e_{2}=e_{3}$\\

$\mathcal{A}_{06}$&$:$&$e_{1} \circ e_{1}=e_{2}$ & $e_{1} \circ e_{2}=e_{3}$\\

$\mathcal{A}_{07}$&$:$&$e_{1} \circ e_{1}=e_{1}$&$e_{2}\circ e_{2}=e_{2}$&$e_{3}  \circ e_{3}=e_{3}$\\

$\mathcal{A}_{08}$&$:$&$e_{1} \circ e_{1}=e_{1}$&$e_{2}\circ e_{2}=e_{2}$&$e_{2} \circ e_{3}=e_{3}$\\

$\mathcal{A}_{09}$&$:$&$e_{1} \circ e_{1}=e_{1} $&$e_{1}\circ e_{2}=e_{2} $&$ e_{1} \circ e_{3}=e_{3}$\\

$\mathcal{A}_{10}$&$:$&$e_{1} \circ e_{1}=e_{1}$&$e_{1} \circ e_{2}=e_{2}$&$e_{1} \circ e_{3}=e_{3}$&$e_{2} \circ e_{2}=e_{3}$\\

$\mathcal{A}_{11}$&$:$&$e_{1} \circ e_{1}=e_{1}$&$e_{2} \circ e_{2}=e_{3}$\\
\end{longtable}
\end{lemma}

\begin{theorem}
\label{3-dim nearly}
Let $\mathcal{A}$ be a nontrivial  $3$-dimensional shift associative algebra. Then $\mathcal{A}$ is 
a  commutative associative  algebra given in Lemma \ref{3-dim asscom}, or it is isomorphic to one of the following algebras:

\begin{longtable}{ccllll}

 $\mathfrak{a}_{1}$&$:$&$e_{1}  e_{2}=e_{3}$&$ e_{2}  e_{1}=-e_{3}$\\

$\mathfrak{a}_{2}^{\alpha }$&$:$&$e_{1}  e_{1}=e_{3}$&$e_{1} e_{2}=e_{3}$&$e_{2}  e_{1}=-e_{3}$&$e_{2}  e_{2}=\alpha e_{3}$\\
\end{longtable}
\end{theorem}

\begin{proof}
 If $(\mathcal{A}, [\cdot, \cdot])=\mathcal{L}_{1}$, 
  then $\mathrm{Z}_{\mathrm{SA}}^{2}\left( \mathcal{L}_{1},\mathcal{L}%
_{1}\right) \neq \mathcal{\varnothing }$ and therefore $\mathcal{L}_{1}$
has shift associative structures. Now, choose an arbitrary element 
\begin{center}
    $\theta
=\left( B_{1},B_{2},B_{3}\right) \in \mathrm{Z}_{\mathrm{SA}}^{2}\left( 
\mathcal{L}_{1},\mathcal{L}_{1}\right) $.
\end{center} Then $\theta =\left(
0,0,\alpha \Delta _{1,1}+\beta \Delta _{1,2}+\gamma \Delta _{2,2}\right) $
for some $\alpha ,\beta ,\gamma \in \mathbb{C}$. Furthermore, an
automorphism $\phi \in \text{Aut}\left( \mathcal{L}_{1}\right) $, is given
by an invertible matrix of the form:%
\[
\left( 
\begin{array}{ccc}
a_{11} & a_{12} & 0 \\ 
a_{21} & a_{22} & 0 \\ 
a_{31} & a_{32} & a_{11}a_{22}-a_{12}a_{21}%
\end{array}%
\right).
\]%
Write $\theta \ast \phi =\left( 0,0,\alpha ^{\prime }\Delta _{1,1}+\beta
^{\prime }\Delta _{1,2}+\gamma ^{\prime }\Delta _{2,2}\right) $. Then%
\begin{longtable}{lcl}
$\alpha ^{\prime }$ &$=$&$\frac{\alpha
a_{11}^{2}+2\beta a_{11}a_{21}+\gamma a_{21}^{2}}{a_{11}a_{22}-a_{12}a_{21}},$ \\
$\beta ^{\prime }$ &$=$&$\frac{\alpha
a_{11}a_{12}+\beta a_{11}a_{22}+\beta a_{12}a_{21}+\gamma
a_{21}a_{22}}{a_{11}a_{22}-a_{12}a_{21}},$ \\
$\gamma ^{\prime }$ &$=$&$\frac{\alpha
a_{12}^{2}+2\beta a_{12}a_{22}+\gamma a_{22}^{2}}{a_{11}a_{22}-a_{12}a_{21}}.$
\end{longtable}%
Then $\left( 
\begin{array}{cc}
\alpha ^{\prime } & \beta ^{\prime } \\ 
\beta ^{\prime } & \gamma ^{\prime }%
\end{array}%
\right) =\frac{1}{a_{11}a_{22}-a_{12}a_{21}}\phi ^{t}\left( 
\begin{array}{cc}
\alpha  & \beta  \\ 
\beta  & \gamma 
\end{array}%
\right) \phi $. Thus, up to a scalar, $\left( 
\begin{array}{cc}
\alpha ^{\prime } & \beta ^{\prime } \\ 
\beta ^{\prime } & \gamma ^{\prime }%
\end{array}%
\right) ,\left( 
\begin{array}{cc}
\alpha  & \beta  \\ 
\beta  & \gamma 
\end{array}%
\right) $ are equivalent. Since $\left( 
\begin{array}{cc}
\alpha  & \beta  \\ 
\beta  & \gamma 
\end{array}%
\right) $ is symmetric, we may assume without any loss of generality that $%
\beta =0$. Then we have the following cases:

\begin{itemize}
\item  $\left( \alpha ,\gamma \right) =\left( 0,0\right) $. Then we get the
algebra $\mathfrak{a}_{1}.$

\item $\left( \alpha ,\gamma \right) \neq \left( 0,0\right) $. Let $\phi$ be the
first of the following matrices if $\alpha \neq 0$, or the second one if $\alpha
=0$:
\[
\left( 
\begin{array}{ccc}
1 & 0 & 0 \\ 
0 & \alpha  & 0 \\ 
0 & 0 & \alpha 
\end{array}%
\right) ,\left( 
\begin{array}{ccc}
0 & -\gamma & 0 \\ 
1 & 0 & 0 \\ 
0 & 0 & \gamma%
\end{array}%
\right).
\]%
Then $\theta \ast \phi =\left( 0,0,\Delta _{1,1}+\gamma ^{\prime }\Delta
_{2,2}\right) $. Hence we get the algebras:%
\begin{longtable}{ccllll}
$\mathfrak{a}_{2}^{\alpha }$&$:$&$e_{1}  e_{1}=e_{3}$&$e_{1} e_{2}=e_{3}$&$e_{2}  e_{1}=-e_{3}$&$e_{2}  e_{2}=\alpha e_{3}$
\end{longtable}
Moreover, the algebras $\mathfrak{a}_{2}^{\alpha }$ and $\mathfrak{a}%
_{2}^{\alpha ^{\prime }}$ are isomorphic if and only if $\alpha =\alpha
^{\prime }$.
\end{itemize}
\end{proof}

\noindent {\bf Classification in dimension 4.}

\begin{lemma}
Let $\mathcal{L}$ be a nontrivial  $4$-dimensional nilpotent Lie algebra. Then $\mathcal{%
L}$\ is isomorphic to one of the following algebras:

\begin{longtable}{lcllll}
 $\mathcal{L}_{1}$&$:$&$\left[ e_{1},e_{2}\right] =e_{3}$\\

 $\mathcal{L}_{2}$&$:$&$\left[ e_{1},e_{2}\right] =e_{3}$&$\left[ e_{1},e_{3}\right] =e_{4}$\\
\end{longtable}
\end{lemma}

\begin{lemma}\label{4-dim asscom}
Let $\mathcal{A}$ be a nontrivial  $4$-dimensional commutative associative algebra. Then $\mathcal{A}= \mathcal{A}_a \oplus \mathcal{A}_b,$
where $\mathcal{A}_a$ is given in Lemma \ref{3-dim asscom} and $\mathcal{A}_b^2=0,$ or it 
is isomorphic to one of the following algebras:

{\setlength{\tabcolsep}{2pt}\begin{longtable}{lcllllll}

 $\mathcal{A}_{12}$&$:$&$e_{1} \circ e_1 =-e_{3}$ &$e_{1} \circ e_{2}=e_{4}$&$e_{2} \circ e_{2}=e_{3}$\\

 $\mathcal{A}_{13}$&$:$&$e_{1} \circ e_{2}=e_{4}$&$e_{2} \circ e_{2}=e_{3}$\\

 $\mathcal{A}_{14}$&$:$&$e_{1}\circ e_1=e_{4}$&$e_{2} \circ e_2=e_{4}$ & $e_{3} \circ e_3=e_{4}$\\

 $\mathcal{A}_{15}$&$:$&$e_{1} \circ e_1=e_{2}$&$e_{1}  \circ e_{2}=e_{4}$ &$e_{3}\circ e_3=e_{4}$\\ 

$\mathcal{A}_{16}$&$:$&$e_{1} \circ e_1 =e_{2} $&$e_{1}\circ  e_{2}=e_{3}$&$e_{1}  \circ e_{3}=e_{4}$&$e_{2}\circ e_{2}=e_{4}$\\

 $\mathcal{A}_{17}$&$:$&$e_{1} \circ e_1=e_{1}$&$e_{2} \circ e_2=e_{2}$ &$e_{3} \circ e_3=e_{3} $&$ e_{4} \circ e_4=e_{4}$\\

 $\mathcal{A}_{18}$&$:$&$e_{1}\circ e_1=e_{1}$&$e_{2} \circ e_2=e_{2}$&$ e_{3} \circ e_3=e_{3}$&$ e_{3} 
\circ e_{4}=e_{4} $\\

$\mathcal{A}_{19}$&$:$&$e_{1} \circ e_1=e_{1}$&$e_{1}  \circ e_{2}=e_{2}$&$ e_{3} \circ e_3=e_{3}$&$ e_{3} \circ e_{4}=e_{4}$\\

$\mathcal{A}_{20}$&$:$&$e_{1}\circ e_1=e_{1}$&$e_{2} \circ e_2=e_{2}$&$e_{2}  \circ e_{3}=e_{3}$ &$e_{2}\circ  e_{4}=e_{4}$\\

 $\mathcal{A}_{21}$&$:$&$e_{1} \circ e_1=e_{1}$&$e_{2} \circ e_2=e_{2}$&$e_{2}  \circ e_{3}=e_{3}$&$e_{2} \circ e_{4}=e_{4}$&$e_{3}\circ e_3=e_{4}$ \\ 

 $\mathcal{A}_{22}$&$:$&$e_{1} \circ e_1=e_{1}$&$e_{1} \circ e_{2}=e_{2}$&$ e_{1}  \circ e_{3}=e_{3}$&$e_{1} \circ e_{4}=e_{4}$\\

 $\mathcal{A}_{23}$&$:$&$e_{1} \circ e_1=e_{1}$&$e_{1} \circ e_{2}=e_{2}$&$e_{1}  \circ e_{3}=e_{3}$&$e_{1} \circ e_{4}=e_{4}$&$e_{2} \circ e_2=e_{3}$\\

 $\mathcal{A}_{24} $&$:$&$e_{1} \circ e_1=e_{1}$&$e_{1} \circ e_{2}=e_{2}$&$e_{1}  \circ e_{3}=e_{3}$&$ e_{1}  \circ e_{4}=e_{4}$&$e_{2} \circ e_2=e_{3}$ &$e_{2}\circ  e_{3}=e_{4}$\\

$\mathcal{A}_{25}$&$:$&$e_{1} \circ e_1 =e_{1}$ &$e_{1} \circ e_{2}=e_{2}$&$e_{1}  \circ e_{3}=e_{3}$&$e_{1}  \circ e_{4}=e_{4}$&$e_{2} \circ e_2=e_{4}$&$e_{3} \circ e_3=e_{4}$\\

 $\mathcal{A}_{26}$&$:$&$e_{1} \circ e_1=e_{1}$&$e_{2} \circ e_2=e_{2}$&$ e_{3}\circ e_3=e_{4}$\\

$\mathcal{A}_{27}$&$:$&$e_{1} \circ e_1=e_{1}$&$ e_{1}  \circ e_{2}=e_{2}$&$ e_{3} \circ e_3=e_{4}$\\

 $\mathcal{A}_{28}$&$ :$&$e_{1} \circ e_1 =e_{1}$&$ e_{2} \circ e_{3}=e_{4}$\\

 $\mathcal{A}_{29}$&$:$&$ e_{1} \circ e_1=e_{1}$&$e_{2} \circ e_2=e_{3} $&$ e_{2}  \circ e_{3}=e_{4}$

\end{longtable}}
 
\end{lemma}

\begin{theorem}
Let $\mathcal{A}$ be a nontrivial  $4$-dimensional shift associative algebra. Then $\mathcal{A}$ is 
a  commutative associative  algebra given in Lemma \ref{4-dim asscom}, or it is isomorphic to one of the following algebras:

\begin{longtable}{lcllllllllll}
$\mathfrak{a}_{01}$&$ :$&$e_{1}  e_{2}=e_{3}$&$e_{2}  e_{1}=-e_{3}$\\

$\mathfrak{a}_{02}^{\alpha } $&$:$&$e_{1}  e_{1}=e_{3}$&$e_{1} e_{2}=e_{3}$&$e_{2}  e_{1}=-e_{3}$&$e_{2}  e_{2}=\alpha e_{3}$\\

$\mathfrak{a}_{03}$&$ :$&$e_{1}  e_{2}=e_{3}$&$e_{2} e_{1}=-e_{3}$&$e_{1}  e_{4}=e_{3}$&$e_{4}  e_{1}=e_{3}$\\

$\mathfrak{a}_{04}$&$ :$&$e_{1}  e_{2}=e_{3} $&$e_{2} e_{1}=-e_{3} $&$e_{1}  e_{4}=e_{3}$&$e_{4}  e_{1}=e_{3}$&$e_{2}  e_{2}=e_{3}$\\

$\mathfrak{a}_{05} $&$:$&$e_{1}  e_{2}=e_{3}$&$e_{2} e_{1}=-e_{3}$&$e_{4}  e_{4}=e_{3}$\\

$\mathfrak{a}_{06}^{\alpha } $&$:$&$e_{1}  e_{1}=e_{3}$&$e_{1} e_{2}=e_{3} $&$ e_{2}  e_{1}=-e_{3}$&$e_{2}  e_{2}=\alpha e_{3}$&$e_{4}  e_{4}=e_{3}$\\

$\mathfrak{a}_{07}^{\alpha } $&$:$&$e_{1}  e_{1}=e_{4}$&$e_{1}  e_{2}=\left( \alpha
+1\right) e_{3} $&$ e_{2}  e_{1}=\left( \alpha -1\right) e_{3}$\\

$\mathfrak{a}_{08} $&$:$&$e_{1}  e_{1}=e_{4} $&$ e_{1}  e_{2}= e_{3} $&$ e_{2}  e_{1}=- e_{3}$&$e_{2}  e_{2}=e_{3}$ \\

$\mathfrak{a}_{09}$&$:$&$e_{1}  e_{2}=e_{3}+e_{4}$&$ e_{2}  e_{1}=e_{4}-e_{3}$\\

$\mathfrak{a}_{10}^{\alpha }$&$:$&$e_{1}  e_{1}=e_{3}$&$e_{1}  e_{2}=e_{3}+e_{4}$&$e_{2}  e_{1}=e_{4}-e_{3}$&$e_{2}  e_{2}=\alpha e_{3}$\\
 
$\mathfrak{a}_{11} $&$:$&$e_{1}  e_{2}=e_{3} $&$ e_{2}  e_{1}=-e_{3}$&$e_{4}  e_{4}=e_{4}$\\

$\mathfrak{a}_{12}^{\alpha }$&$ :$&$e_{1}  e_{1}=e_{3}$&$e_{1}  e_{2}=e_{3} $&$ e_{2}  e_{1}=-e_{3}$&$e_{2}  e_{2}=\alpha e_{3}$&$e_{4}  e_{4}=e_{4}$\\


$\mathfrak{a}_{13}$&$:$&$e_{1}  e_{1}=e_{4}$&$e_{1}  e_{2}=e_{3}$&$e_{2}  e_{1}=-e_{3}$& \\
&&$e_{2}  e_{2}=e_{3}$&$e_{1}  e_{4}=e_{3}$&$e_{4}  e_{1}=e_{3}$\\


$\mathfrak{a}_{14}$&$:$&$e_{1}  e_{1}=e_{4}$&$e_{1}  e_{2}=e_{3}$&$e_{2}  e_{1}=-e_{3}$&$e_{1}  e_{4}=e_{3}$&$e_{4}  e_{1}=e_{3}$
 
\end{longtable}

\end{theorem}

\begin{proof}
If $(\mathcal{A}, [\cdot,\cdot])=\mathcal{L}_{2}$,
then $\mathrm{Z}_{\mathrm{SA}}^{2}\left( \mathcal{L}_{2},\mathcal{L}_{2}\right) = \mathcal{\varnothing }$ and therefore $\mathcal{L}_{2}$
does not have shift associative structures. Finally, if $(\mathcal{A}, [\cdot,\cdot])=\mathcal{L}_{1}$,
then $\mathrm{Z}_{\mathrm{SA}}^{2}\left( \mathcal{L}_{1},\mathcal{L}_{1}\right) \neq \mathcal{\varnothing }$ and therefore $\mathcal{L}_{1}$
has shift associative structures. Furthermore, an automorphism $\phi \in \text{Aut}\left( \mathcal{L}_{1}\right) $, is
given by an invertible matrix of the form:%
\begin{equation*}
\phi=
\begin{pmatrix}
a_{11} & a_{12} & 0 & 0 \\ 
a_{21} & a_{22} & 0 & 0 \\ 
a_{31} & a_{32} & \delta  & a_{34} \\ 
a_{41} & a_{42} & 0 & a_{44}%
\end{pmatrix}%
, \mbox{ \ where \ }\delta =a_{11}a_{22}-a_{12}a_{21}.
\end{equation*}

Now, choose an arbitrary element $\theta
=\left( B_{1},B_{2},B_{3}, B_{4}\right) \in \mathrm{Z}_{\mathrm{SA}}^{2}\left( 
\mathcal{L}_{1},\mathcal{L}_{1}\right) $. Then 
\begin{center}
    $B_{1}=B_{2}=0$ and $\left( B_{3},B_{4}\right) \in \{\left( \vartheta
_{i},\eta _{i}\right)\}_{ 1 \leq i \leq 8},$
\end{center}where

\begin{enumerate}

\item $(\vartheta _{1}, \eta _{1}) \ =\ \big(\alpha _{1}\Delta _{1,1}+\alpha _{2}\Delta
_{1,2}+\alpha _{3}\Delta _{1,4}+\alpha _{4}\Delta _{2,2}+\alpha _{5}\Delta
_{2,4}+\alpha _{6}\Delta _{4,4}; \ 0\big).$

\item $(\vartheta _{2}, \eta _{2}) \ =\ \big(\alpha _{1}\Delta _{1,1}+\alpha _{2}\Delta
_{1,2}+\alpha _{3}\Delta _{2,2}; \ \alpha _{4}\Delta _{1,1}+\alpha
_{5}\Delta _{1,2}+\alpha _{6}\Delta _{2,2}\big)_{\eta _{2}\neq 0}.$

\item $(\vartheta _{3}, \eta _{3}) \ =\ \big(\alpha _{1}\Delta _{1,1}+\alpha _{2}\Delta
_{1,2}+\alpha _{3}\Delta _{2,2}+\alpha _{4}\Delta _{4,4};\ \alpha
_{5}\Delta _{4,4}\big)_{\alpha_5\neq 0}.$

\item $(\vartheta _{4}, \eta _{4}) \ =\ \big(\alpha _{1}\Delta _{1,1}+\alpha _{2}\Delta
_{1,2}+\alpha _{3}\Delta _{1,4}+\alpha _{4}\Delta _{2,2}+\frac{\alpha
_{3}\alpha _{6}}{\alpha _{5}}\Delta _{4,4}; \ \alpha _{5}\Delta
_{1,1}+\alpha _{6}\Delta _{1,4}+\frac{\alpha _{6}^{2}}{\alpha _{5}}\Delta
_{4,4}\big)_{\alpha _{5}\neq 0}$.

\item $(\vartheta _{5}, \eta _{5}) \ =\ \big(\alpha _{1}\Delta _{1,1}+\alpha _{2}\Delta
_{1,2}+\alpha _{3}\Delta _{1,4}+\alpha _{4}\Delta _{2,2}+\alpha _{5}\Delta
_{2,4}+\frac{\alpha _{3}\alpha _{7}}{\alpha _{6}}\Delta _{4,4};$ 
   \begin{flushright}$\frac{\alpha _{3}\alpha _{6}}{\alpha _{5}}\Delta _{1,1}+\alpha _{6}\Delta
_{1,2}+\frac{\alpha _{3}\alpha _{7}}{\alpha _{5}}\Delta _{1,4}+\frac{\alpha
_{5}\alpha _{6}}{\alpha _{3}}\Delta _{2,2}+\alpha _{7}\Delta _{2,4}+\frac{%
\alpha _{3}\alpha _{7}^{2}}{\alpha _{5}\alpha _{6}}\Delta _{4,4}\big)_{
\alpha _{3}\alpha _{5}\alpha _{6}\neq 0}$.  \end{flushright}

\item $(\vartheta _{6}, \eta _{6}) \ =\ \big(\alpha _{1}\Delta _{1,1}+\alpha _{2}\Delta
_{1,2}+\alpha _{3}\Delta _{2,2}+\alpha _{4}\Delta _{2,4}+\frac{\alpha
_{4}\alpha _{6}}{\alpha _{5}}\Delta _{4,4}; \ \alpha _{5}\Delta
_{2,2}+\alpha _{6}\Delta _{2,4}+\frac{\alpha _{6}^{2}}{\alpha _{5}}\Delta
_{4,4}\big)_{\alpha _{5}\neq 0}$.

\item $(\vartheta _{7}, \eta _{7}) \ =\ \big(\alpha _{1}\Delta _{1,1}+\alpha _{2}\Delta
_{1,2}+\alpha _{3}\Delta _{2,2};$  \begin{flushright}$ \frac{\alpha _{4}\alpha _{5}}{%
\alpha _{6}}\Delta _{1,1}+\alpha _{4}\Delta _{1,2}+\alpha _{5}\Delta _{1,4}+%
\frac{\alpha _{4}\alpha _{6}}{\alpha _{5}}\Delta _{2,2}+\alpha _{6}\Delta
_{2,4}+\frac{\alpha _{5}\alpha _{6}}{\alpha _{4}}\Delta _{4,4}\big)_{\alpha
_{4}\alpha _{5}\alpha _{6}\neq 0}$.\end{flushright}

\item $(\vartheta _{8}, \eta _{8}) \ =\ \big(\alpha _{1}\Delta _{1,1}+\alpha _{2}\Delta
_{1,2}+\alpha _{3}\Delta _{2,2}; \  \alpha _{4}\Delta _{2,2}+\alpha
_{5}\Delta _{2,4}+\frac{\alpha _{5}^{2}}{\alpha _{4}}\Delta _{4,4}\big)_{\alpha _{4}\neq 0}$.

\end{enumerate}

Note that each $2$-step nilpotent algebra is shift associative.
The algebraic classification of $4$-dimensional noncommutative $2$-step nilpotent algebras can be found in \cite{kppv}; it gives the algebras $\mathfrak{a}_{01}$--$\mathfrak{a}_{10}^\alpha.$
Hence, we consider the non-$2$-step nilpotent case. 

Then we have the following cases:

\begin{enumerate}
\item $\left( B_{3},B_{4}\right) =\left( \vartheta _{1},\eta _{1}\right) .$ Then $(\mathcal{L}_{1},\cdot _{\theta })$ is a $2$-step nilpotent algebra since $\mathcal{L}_{1}\cdot _{\theta }(\mathcal{L}_{1}\cdot _{\theta }\mathcal{L}_{1})=(\mathcal{L}_{1}\cdot _{\theta }\mathcal{L}_{1})\cdot _{\theta }\mathcal{L}_{1}=0$.

\item $\left( B_{3},B_{4}\right) =\left( \vartheta _{2},\eta _{2}\right) $. Then $(\mathcal{L}_{1},\cdot _{\theta })$ is a $2$-step nilpotent algebra.

\item $\left( B_{3},B_{4}\right) =\left( \vartheta _{3},\eta _{3}\right) $. Then $(\mathcal{L}_{1},\cdot _{\theta })$ is not a $2$-step nilpotent algebra since $e_{4}\cdot _{\theta }e_{4}=\alpha_{5}e_{4}$ and $\alpha_{5} \neq 0$.  
Let $\phi $ be the following automorphism:%
\begin{equation*}
\phi =%
\begin{pmatrix}
a_{11} & a_{12} & 0 & 0 \\ 
a_{21} & a_{22} & 0 & 0 \\ 
0 & 0 & a_{11}a_{22}-a_{12}a_{21} & \frac{\alpha _{4}}{\alpha _{5}^{2}} \\ 
0 & 0 & 0 & \frac{1}{\alpha _{5}}%
\end{pmatrix}%
.
\end{equation*}%
Then $\theta \ast \phi =\left( 0,0,\alpha ^{\prime }\Delta _{1,1}+\beta
^{\prime }\Delta _{1,2}+\gamma ^{\prime }\Delta _{2,2},\Delta _{4,4}\right) $
where $\alpha ^{\prime },\beta ^{\prime },\gamma ^{\prime }$ are as given in
the proof of Theorem \ref{3-dim nearly}. So we get the algebras:%

\begin{longtable}{lcllllllllll} 
 
$\mathfrak{a}_{11} $&$:$&$e_{1}  e_{2}=e_{3} $&$ e_{2}  e_{1}=-e_{3}$&$e_{4}  e_{4}=e_{4}$\\

$\mathfrak{a}_{12}^{\alpha }$&$ :$&$e_{1}  e_{1}=e_{3}$&$e_{1}  e_{2}=e_{3} $&$ e_{2}  e_{1}=-e_{3}$&$e_{2}  e_{2}=\alpha e_{3}$&$e_{4}  e_{4}=e_{4}$\\

\end{longtable}

\item $\left( B_{3},B_{4}\right) =\left( \vartheta _{4},\eta _{4}\right) $. Then $(\mathcal{L}_{1},\cdot _{\theta })$ is a $2$-step nilpotent algebra if $\alpha_{3}=\alpha_{6}=0$. So we may assume $\left(\alpha_{3},\alpha_{6}\right) \neq \left(0,0\right)$.

\begin{itemize}

\item $\alpha _{6}\neq 0$. We choose $\phi $ as follows: 
\begin{equation*}
\phi =%
\begin{pmatrix}
1 & 0 & 0 & 0 \\ 
0 & 1 & 0 & 0 \\ 
0 & 0 & 1 & 0 \\ 
-\frac{\alpha _{5}}{\alpha _{6}} & 0 & 0 & 1%
\end{pmatrix}%
.
\end{equation*}%
Then $\theta \ast \phi =\left( 0,0,\beta _{1}\Delta _{1,1}+\beta _{2}\Delta
_{1,2}+\beta _{3}\Delta _{2,2}+\beta _{4}\Delta _{4,4},\beta _{5}\Delta
_{4,4}\right) $ for some $\beta _{1},\beta _{2},\beta _{3},\beta _{4}\in 
\mathbb{C}$ and $\beta _{5}\in \mathbb{C}^{\ast }$. Hence, we are back in
the case $\left( B_{3},B_{4}\right) =\left( \vartheta _{3},\eta _{3}\right) $%
.

\item $\alpha _{6}=0$. Then we may assume $\alpha _{3}\neq 0$. If $\alpha _{4}\neq 0$, we choose $\phi $ as
follows:%
\begin{equation*}
\phi =%
\begin{pmatrix}
\frac{1}{\alpha _{3}\alpha _{4}\alpha _{5}} & 0 & 0 & 0 \\ 
-\frac{\alpha _{2}}{\alpha _{3}\alpha _{4}^{2}\alpha _{5}} & \frac{1}{\alpha
_{3}\alpha _{4}^{2}\alpha _{5}} & 0 & 0 \\ 
0 & 0 & \frac{1}{\alpha _{3}^{2}\alpha _{4}^{3}\alpha _{5}^{2}} & \frac{%
\alpha _{1}\alpha _{4}-\alpha _{2}^{2}}{\alpha _{3}^{2}\alpha _{4}^{3}\alpha
_{5}^{2}} \\ 
0 & 0 & 0 & \frac{1}{\alpha _{3}^{2}\alpha _{4}^{2}\alpha _{5}}%
\end{pmatrix}%
.
\end{equation*}%
Then $\theta \ast \phi =\left( 0,0,\Delta _{1,4}+\Delta _{2,2},\Delta
_{1,1}\right) $. So we get the algebra:

\begin{longtable}{lcllllllllll} 
$\mathfrak{a}_{13}$&$:$&$e_{1}  e_{1}=e_{4}$&$e_{1}  e_{2}=e_{3}$&$e_{2}  e_{1}=-e_{3}$&$e_{2}  e_{2}=e_{3}$&$e_{1}  e_{4}=e_{3}$&$e_{4}  e_{1}=e_{3}$\\

\end{longtable}If $\alpha _{4}=0$, we choose $\phi $ as follows:%
\begin{equation*}
\phi =%
\begin{pmatrix}
1 & 0 & 0 & 0 \\ 
0 & \alpha _{3}\alpha _{5} & 0 & 0 \\ 
0 & 0 & \alpha _{3}\alpha _{5} & \alpha _{1} \\ 
0 & -\alpha _{2}\alpha _{5} & 0 & \alpha _{5}%
\end{pmatrix}%
.
\end{equation*}%
Then $\theta \ast \phi =\left( 0,0,\Delta _{1,4},\Delta _{1,1}\right) $. So
we get the algebra:

\begin{longtable}{lcllllllllll} 

$\mathfrak{a}_{14}$&$:$&$e_{1}  e_{1}=e_{4}$&$e_{1}  e_{2}=e_{3}$&$e_{2}  e_{1}=-e_{3}$&$e_{1}  e_{4}=e_{3}$&$e_{4}  e_{1}=e_{3}$
 
\end{longtable}

\end{itemize}

\item $\left( B_{3},B_{4}\right) =\left( \vartheta _{5},\eta _{5}\right) $. Then $(\mathcal{L}_{1},\cdot _{\theta })$ is not a $2$-step nilpotent algebra.

\begin{itemize}
\item If $\alpha _{7}=0$, we choose $\phi $ as follows%
\begin{equation*}
\phi =%
\begin{pmatrix}
1 & -\frac{\alpha _{5}}{\alpha _{3}} & 0 & 0 \\ 
0 & 1 & 0 & 0 \\ 
0 & 0 & 1 & 0 \\ 
-\frac{\alpha _{1}}{2\alpha _{3}} & -\frac{\alpha _{2}\alpha _{3}-\alpha
_{1}\alpha _{5}}{\alpha _{3}^{2}} & 0 & \frac{\alpha _{3}}{\alpha _{5}}%
\alpha _{6}%
\end{pmatrix}%
.
\end{equation*}%
Then $\theta \ast \phi =\allowbreak \left( 0,0,\beta _{1}\Delta _{2,2}+\beta
_{2}\Delta _{1,4},\Delta _{1,1}\right) $ for some $\beta _{1}\in \mathbb{C}$ and $\beta _{2}\in \mathbb{C}^{\ast }$. Therefore, we are back in the case $\left( B_{3},B_{4}\right)
=\left( \vartheta _{4},\eta _{4}\right) $.

\item If $\alpha _{7}\neq 0$, we choose $\phi $ as follows%
\begin{equation*}
\phi =%
\begin{pmatrix}
1 & 0 & 0 & 0 \\ 
0 & 1 & 0 & 0 \\ 
0 & 0 & 1 & 0 \\ 
-\frac{\alpha _{3}\alpha _{6}}{\alpha _{3}\alpha _{7}} & -\frac{\alpha
_{5}\alpha _{6}}{\alpha _{3}\alpha _{7}} & 0 & 1%
\end{pmatrix}%
.
\end{equation*}%
Then $\theta \ast \phi =\left( 0,0,\beta _{1}\Delta _{1,1}+\beta _{2}\Delta
_{1,2}+\beta _{3}\Delta _{2,2}+\beta _{4}\Delta _{4,4},\beta _{5}\Delta
_{4,4}\right) $ for some $\beta _{1},\beta _{2},\beta _{3}\in \mathbb{C}$ and $\beta _{4},\beta _{5}\in \mathbb{C}^{\ast }$. So, we are back in the case $\left( B_{3},B_{4}\right) =\left(
\vartheta _{3},\eta _{3}\right) $.
\end{itemize}

\item $\left( B_{3},B_{4}\right) =\left( \vartheta _{6},\eta _{6}\right) $. We choose $\phi $ as follows:%
\begin{equation*}
\phi =%
\begin{pmatrix}
0 & 1 & 0 & 0 \\ 
1 & 0 & 0 & 0 \\ 
0 & 0 & -1 & 0 \\ 
0 & 0 & 0 & 1%
\end{pmatrix}%
.
\end{equation*}%
Then we have:
\[\adjustbox{max width=\linewidth}{$\theta \ast \phi =\left( 0,0,-\alpha _{3}\Delta _{1,1}-\alpha _{2}\Delta
_{1,2}-\alpha _{1}\Delta _{2,2}-\alpha _{4}\Delta _{1,4}-\frac{\alpha _{4}\alpha _{6}}{%
\alpha _{5}}\Delta _{4,4},\alpha _{5}\Delta _{1,1}+\alpha
_{6}\Delta _{1,4}+\frac{\alpha _{6}^{2}}{\alpha _{5}}\Delta _{4,4}\right) .$}\]%
Thus, we are back in the case $\left( B_{3},B_{4}\right) =\left( \vartheta
_{4},\eta _{4}\right) $.

\item $\left( B_{3},B_{4}\right) =\left( \vartheta _{7},\eta _{7}\right) $.
Choose $\phi $ as follows:%
\begin{equation*}
\begin{pmatrix}
1 & 0 & 0 & 0 \\ 
0 & 1 & 0 & 0 \\ 
0 & 0 & 1 & 0 \\ 
-\frac{\alpha _{4}}{\alpha _{6}} & -\frac{\alpha _{4}}{\alpha _{5}} & 0 & 1%
\end{pmatrix}%
\end{equation*}%
Then $\theta \ast \phi =\left( 0,0,\beta _{1}\Delta _{1,1}+\beta _{2}\Delta
_{1,2}+\beta _{3}\Delta _{2,2},\beta _{4}\Delta _{4,4}\right) $ for some $%
\beta _{1},\beta _{2},\beta _{3}\in \mathbb{C}
$ and $\beta _{4}\in \mathbb{C}^{\ast }$. Hence, we are back in the case $\left( B_{3},B_{4}\right) =\left(
\vartheta _{3},\eta _{3}\right) $.

\item $\left( B_{3},B_{4}\right) =\left( \vartheta _{8},\eta _{8}\right) $.
We may assume $\alpha _{5}\neq 0$ since otherwise we are in the case $\left(
B_{3},B_{4}\right) =\left( \vartheta _{6},\eta _{6}\right) $. Choose $\phi $
as follows:%
\begin{equation*}
\phi =%
\begin{pmatrix}
1 & 0 & 0 & 0 \\ 
0 & 1 & 0 & 0 \\ 
0 & 0 & 1 & 0 \\ 
0 & -\frac{\alpha _{4}}{\alpha _{5}} & 0 & 1%
\end{pmatrix}%
.
\end{equation*}%
Then $\theta \ast \phi =\left( 0,0,\beta _{1}\Delta _{1,1}+\beta _{2}\Delta
_{1,2}+\beta _{3}\Delta _{2,2},\beta _{4}\Delta _{4,4}\right) $ for some $%
\beta _{1},\beta _{2},\beta _{3}\in \mathbb{C} $ and $\beta _{4}\in \mathbb{C}^{\ast }$. Hence, we are back in the case $\left( B_{3},B_{4}\right) =\left(
\vartheta _{3},\eta _{3}\right) $.

\end{enumerate}

\end{proof}

\begin{corollary}\label{4dim}
    Each $4$-dimensional shift associative algebra is 
    cyclic associative.
\end{corollary}

A non-associative shift associative algebra of dimension $6$ is given in
\cite[Example 3.17]{BBR}. The next observation corrects the bound of the minimal dimension of non-associative shift associative algebras.

\begin{proposition}
    
        The minimal dimension of a non-associative
    shift associative algebra is $5.$  
\end{proposition}

\begin{proof}
Thanks to Corollary \ref{4dim} and \cite[Example 3.17]{BBR}, 
we have to find a $5$-dimensional non-associative shift associative algebra.
Such an example is given by the following multiplication table:
\begin{longtable}{rclrclrclrcl}
$e_1 e_1 =e_3$ &
$e_1e_2 =e_4$ & 
$e_2 e_3 = e_5$& 
$ e_4e_1= e_5$
\end{longtable}

\end{proof}

\subsection{The geometric classification of shift associative  algebras}

\
The study of varieties of non-associative algebras from a geometric point of view has a long history 
(see \cite{l24,MS,k23,fkkv,DotsHij,GRH,Albert1,GRH3,akm,KM14,kppv,BC99,ale,gkp,ikv17} and references therein).

\noindent {\bf Definitions and notation.}
Given an $n$-dimensional vector space $\mathbb V$, the set ${\rm Hom}(\mathbb V \otimes \mathbb V,\mathbb V) \cong \mathbb V^* \otimes \mathbb V^* \otimes \mathbb V$ is a vector space of dimension $n^3$. This space has the structure of the affine variety $\mathbb{C}^{n^3}.$ Indeed, let us fix a basis $e_1,\dots,e_n$ of $\mathbb V$. Then any $\mu\in {\rm Hom}(\mathbb V \otimes \mathbb V,\mathbb V)$ is determined by $n^3$ structure constants $c_{ij}^k\in\mathbb{C}$ such that
$\mu(e_i\otimes e_j)=\sum\limits_{k=1}^nc_{ij}^ke_k$. A subset of ${\rm Hom}(\mathbb V \otimes \mathbb V,\mathbb V)$ is {\it Zariski-closed} if it can be defined by a set of polynomial equations in the variables $c_{ij}^k$ ($1\le i,j,k\le n$).

Let $T$ be a set of polynomial identities.
The set of algebra structures on $\mathbb V$ satisfying polynomial identities from $T$ forms a Zariski-closed subset of the variety ${\rm Hom}(\mathbb V \otimes \mathbb V,\mathbb V)$. We denote this subset by $\mathbb{L}(T)$.
The general linear group ${\rm GL}(\mathbb V)$ acts on $\mathbb{L}(T)$ by conjugations:
$$ (g * \mu )(x\otimes y) = g\mu(g^{-1}x\otimes g^{-1}y)$$
for $x,y\in \mathbb V$, $\mu\in \mathbb{L}(T)\subset {\rm Hom}(\mathbb V \otimes\mathbb V, \mathbb V)$ and $g\in {\rm GL}(\mathbb V)$.
Thus, $\mathbb{L}(T)$ is decomposed into ${\rm GL}(\mathbb V)$-orbits that correspond to the isomorphism classes of algebras.
Let ${\mathcal O}(\mu)$ denote the orbit of $\mu\in\mathbb{L}(T)$ under the action of ${\rm GL}(\mathbb V)$ and $\overline{{\mathcal O}(\mu)}$ denote the Zariski closure of ${\mathcal O}(\mu)$.

Let $\bf A$ and $\bf B$ be two $n$-dimensional algebras satisfying the identities from $T$, and let $\mu,\lambda \in \mathbb{L}(T)$ represent $\bf A$ and $\bf B$, respectively.
We say that $\bf A$ degenerates to $\bf B$ and write $\bf A\to \bf B$ if $\lambda\in\overline{{\mathcal O}(\mu)}$.
Note that in this case we have $\overline{{\mathcal O}(\lambda)}\subset\overline{{\mathcal O}(\mu)}$. Hence, the definition of degeneration does not depend on the choice of $\mu$ and $\lambda$. If $\bf A\not\cong \bf B$, then the assertion $\bf A\to \bf B$ is called a {\it proper degeneration}. We write $\bf A\not\to \bf B$ if $\lambda\not\in\overline{{\mathcal O}(\mu)}$.

Let $\bf A$ be represented by $\mu\in\mathbb{L}(T)$. Then  $\bf A$ is  {\it rigid} in $\mathbb{L}(T)$ if ${\mathcal O}(\mu)$ is an open subset of $\mathbb{L}(T)$.
 Recall that a subset of a variety is called irreducible if it cannot be represented as a union of two non-trivial closed subsets.
 A maximal irreducible closed subset of a variety is called an {\it irreducible component}.
It is well known that any affine variety can be represented as a finite union of its irreducible components in a unique way.
The algebra $\bf A$ is rigid in $\mathbb{L}(T)$ if and only if $\overline{{\mathcal O}(\mu)}$ is an irreducible component of $\mathbb{L}(T)$.

\medskip

\noindent {\bf Method of the description of degenerations of algebras.} In the present work we use the methods applied to Lie algebras in \cite{GRH}.
First of all, if $\bf A\to \bf B$ and $\bf A\not\cong \bf B$, then $\mathfrak{Der}(\bf A)<\mathfrak{Der}(\bf B)$, where $\mathfrak{Der}(\bf A)$ is the   algebra of derivations of $\bf A$. We compute the dimensions of algebras of derivations and check the assertion $\bf A\to \bf B$ only for such $\bf A$ and $\bf B$ that $\mathfrak{Der}(\bf A)<\mathfrak{Der}(\bf B)$.

To prove degenerations, we construct families of matrices parametrized by $t$. Namely, let $\bf A$ and $\bf B$ be two algebras represented by the structures $\mu$ and $\lambda$ from $\mathbb{L}(T)$ respectively. Let $e_1,\dots, e_n$ be a basis of $\mathbb  V$ and $c_{ij}^k$ ($1\le i,j,k\le n$) be the structure constants of $\lambda$ in this basis. If there exist $a_i^j(t)\in\mathbb{C}$ ($1\le i,j\le n$, $t\in\mathbb{C}^*$) such that $E_i^t=\sum\limits_{j=1}^na_i^j(t)e_j$ ($1\le i\le n$) form a basis of $\mathbb V$ for any $t\in\mathbb{C}^*$, and the structure constants of $\mu$ in the basis $E_1^t,\dots, E_n^t$ are such rational functions $c_{ij}^k(t)\in\mathbb{C}[t]$ that $c_{ij}^k(0)=c_{ij}^k$, then $\bf A\to \bf B$.
In this case  $E_1^t,\dots, E_n^t$ is called a {\it parametrized basis} for $\bf A\to \bf B$.
In  case of  $E_1^t, E_2^t, \ldots, E_n^t$ is a {\it parametric basis} for ${\bf A}\to {\bf B},$ it will be denoted by
${\bf A}\xrightarrow{(E_1^t, E_2^t, \ldots, E_n^t)} {\bf B}$. 
To simplify our equations, we will use the notation $A_i=\langle e_i,\dots,e_n\rangle,\ i=1,\ldots,n$ and write simply $A_pA_q\subset A_r$ instead of $c_{ij}^k=0$ ($i\geq p$, $j\geq q$, $k< r$).


Let ${\bf A}(*):=\{ {\bf A}(\alpha)\}_{\alpha\in I}$ be a series of algebras, and let $\bf B$ be another algebra. Suppose that for $\alpha\in I$, $\bf A(\alpha)$ is represented by the structure $\mu(\alpha)\in\mathbb{L}(T)$ and $\bf B$ is represented by the structure $\lambda\in\mathbb{L}(T)$. Then we say that $\bf A(*)\to \bf B$ if $\lambda\in\overline{\{{\mathcal O}(\mu(\alpha))\}_{\alpha\in I}}$, and $\bf A(*)\not\to \bf B$ if $\lambda\not\in\overline{\{{\mathcal O}(\mu(\alpha))\}_{\alpha\in I}}$.

Let $\bf A(*)$, $\bf B$, $\mu(\alpha)$ ($\alpha\in I$) and $\lambda$ be as above. To prove $\bf A(*)\to \bf B$ it is enough to construct a family of pairs $(f(t), g(t))$ parametrized by $t\in\mathbb{C}^*$, where $f(t)\in I$ and $g(t)\in {\rm GL}(\mathbb V)$. Namely, let $e_1,\dots, e_n$ be a basis of $\mathbb V$ and $c_{ij}^k$ ($1\le i,j,k\le n$) be the structure constants of $\lambda$ in this basis. If we construct $a_i^j:\mathbb{C}^*\to \mathbb{C}$ ($1\le i,j\le n$) and $f: \mathbb{C}^* \to I$ such that $E_i^t=\sum\limits_{j=1}^na_i^j(t)e_j$ ($1\le i\le n$) form a basis of $\mathbb V$ for any  $t\in\mathbb{C}^*$, and the structure constants of $\mu({f(t)})$ in the basis $E_1^t,\dots, E_n^t$ are such rational functions $c_{ij}^k(t)\in\mathbb{C}[t]$ that $c_{ij}^k(0)=c_{ij}^k$, then $\bf A(*)\to \bf B$. In this case  $E_1^t,\dots, E_n^t$ and $f(t)$ are called a parametrized basis and a {\it parametrized index} for $\bf A(*)\to \bf B$, respectively.

We now explain how to prove $\bf A(*)\not\to\mathcal  \bf B$.
Note that if $\mathfrak{Der} \ \bf A(\alpha)  > \mathfrak{Der} \  \bf B$ for all $\alpha\in I$ then $\bf A(*)\not\to\bf B$.
One can also use the following  Lemma, whose proof is the same as the proof of Lemma 1.5 from \cite{GRH}.

\begin{lemma}\label{gmain}
Let $\mathfrak{B}$ be a Borel subgroup of ${\rm GL}(\mathbb V)$ and $\mathcal{R}\subset \mathbb{L}(T)$ be a $\mathfrak{B}$-stable closed subset.
If $\bf A(*) \to \bf B$ and for any $\alpha\in I$ the algebra $\bf A(\alpha)$ can be represented by a structure $\mu(\alpha)\in\mathcal{R}$, then there is $\lambda\in \mathcal{R}$ representing $\bf B$.
\end{lemma}

\noindent{\bf The geometric classification of shift associative algebras.}
The main result of the present section is Theorem \ref{geo1}.

\begin{proposition}\label{dim3geo}
The variety of complex $3$-dimensional nilpotent shift associative algebras  has 
dimension  $6$   and it has  $2$  irreducible components defined by  
\begin{center}
$\mathcal{C}_1=\overline{\mathcal{O}( \mathcal{A}_{06})}$ and
$\mathcal{C}_2=\overline{\mathcal{O}(\mathfrak{a}_{02}^{\alpha })},$ 
\end{center}
In particular, there is only one rigid algebra in this variety.
The variety of complex $3$-dimensional shift associative algebras  has 
dimension  $9$   and it has  $2$  irreducible components defined by  
\begin{center}
$\mathcal{C}_1=\overline{\mathcal{O}( \mathcal{A}_{07})}$ and
$\mathcal{C}_2=\overline{\mathcal{O}( \mathfrak{a}_{02}^{\alpha })}.$   
\end{center}
In particular, there is only one rigid algebra in this variety.
\end{proposition}
\begin{proof}
Thanks to Theorem \ref{3-dim nearly},
each $3$-dimensional shift associative algebra is 
 commutative associative or $2$-step nilpotent.
The variety of $3$-dimensional nilpotent commutative associative algebras is irreducible and defined by $\mathcal{O}( \mathcal{A}_{06})$ \cite{fkkv};
  the variety of $3$-dimensional commutative associative algebras is irreducible and defined by $\mathcal{O}( \mathcal{A}_{07})$ \cite{gkp};
 the variety of $3$-dimensional $2$-step nilpotent algebras  is irreducible and defined by ${\mathcal{O}(\mathfrak{a}_{02}^{\alpha })}$ \cite{kppv}.
  Comparing the dimensions of these components and taking into account that the closure of a set of commutative algebras consists of commutative algebras, we obtain the statement.
 \end{proof}

\begin{theorem}\label{geo1}
The variety of complex $4$-dimensional nilpotent shift associative algebras  has 
dimension  $12$   and it has  $4$  irreducible components defined by  
\begin{center}
$\mathcal{C}_1=\overline{\mathcal{O}( \mathcal{A}_{16})},$ \
$\mathcal{C}_2=\overline{\mathcal{O}(\mathfrak{a}_{06}^{\alpha })},$ 
$\mathcal{C}_3=\overline{\mathcal{O}(\mathfrak{a}_{10}^{\alpha })}$ \  and
$\mathcal{C}_4=\overline{\mathcal{O}( \mathfrak{a}_{13})}.$
\end{center}
In particular, there are exactly two rigid algebras in this variety: $\mathcal{A}_{16}$ and $\mathfrak{a}_{13}$.
The variety of complex $4$-dimensional shift associative algebras  has 
dimension  $16$   and it has  $3$  irreducible components defined by  
\begin{center}
$\mathcal{C}_1=\overline{\mathcal{O}( \mathcal{A}_{17})},$ \
$\mathcal{C}_2=\overline{\mathcal{O}(\mathfrak{a}_{10}^{\alpha })}$ and
$\mathcal{C}_3=\overline{\mathcal{O}( \mathfrak{a}_{12}^{\alpha })}.$   
\end{center}
In particular, there is only one rigid algebra in this variety.
 
\end{theorem}

\begin{proof}
The dimensions of the orbits of the relevant algebras (for the families, the dimensions of the closures of the unions of the orbits) are 

\begin{center}
      
$\dim  \mathcal{O}(\mathcal{A}_{17})= 16,$ \ 
$\dim \mathcal{O}(\mathfrak{a}_{12}^{\alpha })= 13,$ \ 
$\dim  \mathcal{O}(\mathcal{A}_{16})=  
\dim \mathcal{O}(\mathfrak{a}_{06}^{\alpha })=
\dim \mathcal{O}(\mathfrak{a}_{10}^{\alpha })=
\dim \mathcal{O}(\mathfrak{a}_{13})=12.$

\end{center}

Thanks to \cite{fkkv}, 
each $4$-dimensional   nilpotent commutative associative algebra is in $\overline{\mathcal{O}(\mathcal{A}_{16})}.$ 
Thanks to \cite{kppv}, 
each $4$-dimensional $2$-step nilpotent algebra is in $\mathcal{O}(\mathfrak{a}_{06}^{\alpha })$ or $\mathcal{O}(\mathfrak{a}_{10}^{\alpha }).$ To obtain the geometric classification of $4$-dimensional nilpotent 
shift associative algebras, we need the following degeneration
    ${\mathfrak a}_{13}  \xrightarrow{ (t e_1; \  t^2e_2; \  t^{3}e_3; \   t^2 e_4)} {\mathfrak a}_{14}.$

\ 

Now, we consider all $4$-dimensional shift associative algebras. 
Thanks to \cite{KM14}, 
each $4$-dimensional commutative associative algebra is in $\overline{\mathcal{O}( \mathcal{A}_{17})}.$
Below, we present the rest of the necessary degenerations: 
\begin{center}
    
  ${\mathfrak a}_{12}^{\alpha} \xrightarrow{ (e_1; \ e_2; \ e_3; \ t e_4-t^{-1}e_3)} {\mathfrak a}_{06}^{\alpha},$ \ 
 ${\mathfrak a}_{12}^{0} \xrightarrow{ (t e_1; \  e_2; \  t e_3; \  e_4)} {\mathfrak a}_{11},$ \\  
    ${\mathfrak a}_{12}^{-\frac{1}{t}} \xrightarrow{ (e_1+te_2+te_4; \  (t-1)te_2+t^2e_4; \  t^{3}e_4; \  (1-t)e_3 +t^2 e_4)} {\mathfrak a}_{13}.$ \  
\end{center}

The non-degeneration
 ${\mathfrak a}_{12}^{\alpha} \not\to  {\mathfrak a}_{10}^\beta$
 follows from Lemma \ref{gmain} applied to the following $\mathfrak{B}$-stable closed subset $\mathcal{R}$, which contains a representative of each algebra ${\mathfrak a}_{12}^{\alpha}$ and no representative of the algebras ${\mathfrak a}_{10}^\beta$:

$\mathcal R=\left\{  
\begin{array}{llll}
A_1^2 \subseteq A_3,\\
(c_{12}^4)^2 (c_{21}^3)^2+ (c_{12}^3)^2 (c_{21}^4)^2 
 - 
   c_{11}^3  c_{22}^3\big( (c_{12}^4)^2 + (c_{21}^4)^2\big)  
  +\\ 
   2 c_{12}^4 c_{21}^4(c_{11}^3  c_{22}^3 -   c_{12}^3  c_{21}^3 )  +  
   c_{11}^3  c_{22}^4 (c_{12}^3 c_{12}^4  -   c_{12}^4 c_{21}^3)  +
  c_{11}^3  c_{21}^4 c_{22}^4  (c_{21}^3-c_{12}^3) 
   =\\  
 \multicolumn{1}{r}{= \ c_{11}^4 (c_{12}^3 - c_{21}^3) \big( c_{22}^3 (c_{21}^4-c_{12}^4) +  c_{22}^4(c_{12}^3 - c_{21}^3)\big)}

 \end{array}
 \right\}$

\end{proof}

Several conjectures state that nilpotent Lie algebras form a very small subvariety in the variety of Lie algebras. Grunewald and O'Halloran conjectured in \cite{GRH3} that for any $n$-dimensional nilpotent Lie algebra $\mathcal A$ there exists an $n$-dimensional non-nilpotent Lie algebra $\mathcal B$ such that $\mathcal B\to \mathcal A$. At the same time,
Vergne conjectured in \cite{V70} that a nilpotent Lie algebra cannot be rigid in the variety of all Lie algebras. Analogous assertions can be conjectured for other varieties.
We will say that the variety $\mathfrak{A}$ of algebras has {\it Grunewald--O'Halloran Property} if for any nilpotent algebra $\mathcal A\in\mathfrak{A}$ there is a non-nilpotent algebra $\mathcal B\in\mathfrak{A}$ such that $\mathcal B\to \mathcal  A$.
We will say that $\mathfrak{A}$ has {\it Vergne Property} if there are no nilpotent rigid algebras in $\mathfrak{A}$.
We will also say that $\mathfrak{A}$ has {\it Vergne--Grunewald--O'Halloran Property} if any irreducible component of $\mathfrak{A}$ contains a non-nilpotent algebra.
Grunewald--O'Halloran Property was proved for four-dimensional Jordan and  Lie algebras in \cite{KM14,BC99}.
It is clear that Vergne--Grunewald--O'Halloran Property follows from Grunewald--O'Halloran Property and Vergne Property follows from Vergne--Grunewald--O'Halloran Property.
It was proven that the variety of complex $4$-dimensional (one-sided) Leibniz algebras has Vergne--Grunewald--O'Halloran Property and on the other hand, it does not have Grunewald--O'Halloran Property \cite{ikv17}.
At the same time, 
the variety of symmetric Leibniz algebras has no Vergne--Grunewald--O'Halloran Property; however, it has  Vergne Property \cite{ale}.
As a corollary of Theorem \ref{geo1}, we have the following statement.

\begin{corollary}
The variety of complex $4$-dimensional shift associative algebras has no Vergne--Grunewald--O'Halloran Property. However, it has  Vergne Property.
\end{corollary}

 \begin{proposition}\label{rigidalg}
  Let $\mathcal A$ be an $n$-dimensional semisimple commutative associative algebra.
 Then $\mathcal A$ is rigid in the variety of $n$-dimensional shift associative algebras.
 \end{proposition}

\begin{proof}
        Let $\mathcal{U}$ be the set of algebra structures $\mu$ such that the trace form $(x,y)\mapsto {\rm tr}\, L^{+}_{x\circ y}$ of the algebra $(\mathbb V,\mu)^{+}$ is nondegenerate; it is an open subset containing a structure representing $\mathcal{A}$.
    If $\mu\in\mathcal{U}$ is a shift associative structure, then the Jordan algebra $(\mathbb V,\mu)^{+}$ has a nondegenerate trace form; hence, it is semisimple.
    By Corollary \ref{radeq} and Theorem \ref{sum}, $(\mathbb V,\mu)$ is a semisimple commutative associative algebra; hence, it is isomorphic to $\mathcal A$.
    Thus, the orbit of $\mathcal A$ is open in the variety of $n$-dimensional shift associative algebras.
\end{proof}

Theorem \ref{geo1} and Propositions \ref{dim3geo} and \ref{rigidalg} motivate the following open question.

\ 

\noindent{\bf Open question.}
Is there a rigid non-associative shift associative algebra?

\


\ 


\ 

\noindent

\ 

\noindent

\ 

\noindent

\appendix
\section*{Appendix A. Differences from the published version}
\addcontentsline{toc}{section}{Appendix A. Differences from the published version}

\noindent\textbf{Lemma \ref{lemknice}:} the proof has been corrected.

\noindent\textbf{Theorem \ref{table}:} the list of identities of the anticommutator of shift associative algebras has been corrected, and the proof has been completed.

\noindent\textbf{Lemma \ref{all0}:} new identities have been added.

\noindent\textbf{Theorems \ref{teofree} and \ref{teofreec}:} the sets $\mathcal{B}_4$ (for $\SAs(X)$) and $\mathcal{B}_3$ (for $\CAs(X)$) have been corrected: the published sets are not linearly independent if some indices coincide.

\end{document}